\documentclass[11pt]{amsart}
\usepackage[totalwidth=440pt, totalheight=640pt]{geometry}
\usepackage{amsmath, amssymb, amsthm, amsfonts, mathrsfs, eucal, enumerate, natbib, layout}
\usepackage[normalem]{ulem}
\usepackage[all,tips,2cell]{xy}
\usepackage[usenames]{color}
\usepackage{verbatim}
\usepackage{mathtools}

\newcommand{\details}[1]{}

\newtheorem{theorem}{Theorem}[section]
\newtheorem*{theorem*}{Theorem}
\newtheorem{corollary}[theorem]{Corollary}
\newtheorem*{corollary*}{Corollary}
\newtheorem{lemma}[theorem]{Lemma}

\newtheorem*{claim*}{Claim}
\newtheorem*{lemma*}{Lemma}
\newtheorem{proposition}[theorem]{Proposition}
\newtheorem*{proposition*}{Proposition}
\newtheorem{conjecture}[theorem]{Conjecture}
\newtheorem*{conjecture*}{Conjecture}
\newtheorem{def-proposition}[theorem]{Definition-Proposition}
\theoremstyle{definition}

\newtheorem*{definition*}{Definition}
\newtheorem{remark}[theorem]{Remark}

\newtheorem*{example*}{Example}
\numberwithin{equation}{section}

\newcommand{\pn}{\noindent}

\newcommand{\ZZ}{\mathbb{Z}}
\newcommand{\QQ}{\mathbb{Q}}

\newcommand{\CC}{\mathbb{C}}
\newcommand{\GG}{\mathbb{G}}
\newcommand{\PP}{\mathbb{P}}

\newcommand{\End}{\mathrm{End}}
\newcommand{\Hom}{\mathrm{Hom}}
\newcommand{\Ext}{\mathrm{Ext}}

\newcommand{\uHom}{\underline{\mathrm{Hom}}}

\newcommand{\UR}{\mathrm{UR}}
\newcommand{\W}{\mathrm{W}}
\newcommand{\F}{\mathrm{F}}
\newcommand{\T}{\mathrm{T}}

\newcommand{\HH}{\mathrm{H}}

\newcommand{\Lie}{\mathrm{Lie}}

\newcommand{\Gr}{\mathrm{Gr}}
\newcommand{\rk}{\mathrm{rk}}

\newcommand{\id}{\mathrm{id}}

\newcommand{\dR}{\mathrm{dR}}

\newcommand{\Galmot}{{\mathcal{G}}{\mathrm{al}}_{\mathrm{mot}}}

\newcommand{\Ind}{{\mathrm{Ind}}\,}
\def\Qbar{\overline{\QQ}}

\newcommand{\oK}{\overline K}
\newcommand{\gal}{{\mathrm{Gal}}(\oK/K)}

\newcommand{\cE}{\mathcal{E}}
\newcommand{\im}{\mathrm{Im}}
\newcommand{\pr}{\mathrm{pr}}

\begin{document}

\title[Third kind elliptic integrals and 1-motives]
{Third kind elliptic integrals and 1-motives}

\author{Cristiana Bertolin}
\address{Dipartimento di Matematica, Universit\`a di Torino, Via Carlo Alberto 10, Italy}
\email{cristiana.bertolin@unito.it}



\dedicatory{with a letter of Y. Andr\'e \\ and an appendix by M. Waldschmidt }


\begin{abstract}

In \cite{B02} we have showed that the Generalized Grothendieck's Period Conjecture applied to 1-motives, whose underlying semi-abelian variety is a product of elliptic curves and of tori,
is equivalent to a transcendental conjecture involving elliptic integrals of the first and second kind, and logarithms of complex numbers.

In this paper we investigate the Generalized Grothendieck's Period Conjecture in the case of 1-motives whose underlying semi-abelian variety is a \textit{non trivial extension} of a product of elliptic curves by a torus. This will imply the introduction of \textit{elliptic integrals of the third kind} for the computation of the periods of $M$ and therefore the Generalized Grothendieck's Period Conjecture applied to $M$ will be equivalent to a transcendental conjecture involving elliptic integrals of the first, second and third kind.

\end{abstract}


\maketitle


\tableofcontents

\section*{Introduction}

Let $\cE$ be an elliptic curve defined over $\CC$ with Weierstrass coordinate functions $x$ and $y$. On $\cE$ we have the  differential of the first kind $\omega = \frac{dx}{y},$ which is holomorphic, the differential of the second kind $
	\eta = -\frac{xdx}{y},$
 which has a double pole with residue zero at each point of the lattice $\HH_1(\cE(\CC),\ZZ)$ and no other pole, and the differential of the third kind
 \[	\xi_Q = \frac{1}{2} \frac{y-y(Q)}{x - x(Q)} \frac{dx}{y}, \]
	for any point $Q $ of $ \cE(\CC), Q \not=0,$ whose residue divisor is $D=-(0)+(-Q).$
Let $\gamma_1, \gamma_2$  be two closed paths on $\cE(\CC)$ which build a basis for the lattice $\HH_1(\cE(\CC),\ZZ)$.
In his Peccot lecture 
at the Coll\`ege de France in 1977, M. Waldschmidt observed that  the periods of the Weierstrass $\wp$-function (\ref{eq:periods-wp}) are the elliptic integrals of the first kind $ \int_{\gamma_i} \omega = \omega_i$  $(i=1,2)$,
the quasi-periods of the Weierstrass $\zeta$-function (\ref{eq:periods-zeta}) are the elliptic integrals of the second kind $ \int_{\gamma_i} \eta = \eta_i$ $(i=1,2)$,
 but \textit{there is no function whose quasi-quasi-periods are elliptic integrals of the third kind}. J.-P.~Serre answered this question furnishing the function
\[
f_q(z)= \frac{\sigma(z+q)}{\sigma(z) \sigma(q)} e^{-\zeta(q) z }  \qquad \mathrm{with}\; q \in \CC \setminus \Lambda
\] 
whose \textit{quasi-quasi periods} (\ref{eq:periods-fq}) are 
\textit{the exponentials of the elliptic integrals of the third kind}
$ \int_{\gamma_i} \xi_Q = \eta_i q - \omega_i \zeta(q)$ $(i=1,2),$ where $q$ is an elliptic logarithm of the point $Q$.

Consider now an extension $G$ of $\cE$ by $\GG_m$ parameterized by the divisor $D=(-Q)-(0)$ of $\mathrm{Pic}^0(\cE) \cong \cE^* = \underline{\Ext}^1(\cE,\GG_m)$. Since the three differentials $\{\omega, \eta,\xi_Q\}$ build a basis of the De Rham cohomology $\HH^1_{\dR}(G)$ of the extension $G$, elliptic integrals of the third kind play a role in the
 Generalized Grothendieck's Period Conjecture (\ref{eq:GCP}). The aim of this paper is to understand this role applying the Generalized Grothendieck's Period Conjecture to 1-motives whose underlying semi-abelian variety is a \textit{non trivial extension} of a product of elliptic curves by a torus.
At the end of this paper the reader can find 
\begin{itemize}
	\item an appendix by M. Waldschmidt in which he quotes transcendence results concerning elliptic integrals of the third kind;
	\item a letter of Y. Andr\'e containing an overview of Grothendieck's Period Conjecture and its generalization.
\end{itemize}

A 1-motive $M=[u:X \rightarrow G]$ over a sub-field $K$ of $\CC$ consists of a finitely generated free $\ZZ$-module $X$, an extension $G$ of an abelian variety by a torus, and a homomorphism $u:X \to G(K)$. Denote by $M_\CC $ the 1-motive defined over $\CC$ obtained from $M$ extending the scalars from $K$ to $\CC$. In \cite{D75} Deligne associates to the 1-motive $M$ 
\begin{itemize}
	\item its De Rham realization $\T_{\dR}(M)$: it is the finite dimensional $K$-vector space $\Lie (G^\natural)$, with $ M^\natural =[u:X \rightarrow G^\natural]$ the universal extension of $M$ by the vector group 
	\par\noindent 
	$\Hom(\Ext^1(M,\GG_a),\GG_a)$,
	\item its Hodge realization $\T_{\QQ}(M_\CC)$: it is the finite dimensional $\QQ$-vector space $\T_{\ZZ}(M_\CC) \otimes_\ZZ \QQ$, with $\T_{\ZZ}(M_\CC)$ the fibered product of $\Lie (G)$ and $X$ over $G$ via the exponential map $ \exp : \Lie (G) \to G$ and the homomorphism $u:X \to G.$ The $\ZZ$-module $\T_{\ZZ}(M_\CC)$ is in fact endowed with a structure of $\ZZ$-mixed Hodge structure, without torsion, of level $\leq 1$, and of type $\{(0,0),(0,-1),(-1,0), (-1,-1)\}.$   
\end{itemize} 
Since the Hodge realizations attached to 1-motives are mixed Hodge structures, 1-motives are mixed motives.
In particular they are the mixed motives coming geometrically from varieties of dimension $\leq 1$. 
In \cite[(10.1.8)]{D75}, Deligne shows that the De Rham and the Hodge realizations of $M$ are isomorphic
\begin{equation} \label{eq:betaM}
	\beta_M: \T_{\dR}(M) \otimes_K \CC  \longrightarrow \T_{\QQ}(M_\CC)\otimes_K \CC.
\end{equation}
The \textit{periods of M} are the coefficients of the matrix which represents 
this isomorphism with respect to $K$-bases.

By Nori's and Ayoub's works (see \cite{Ay14} and \cite{N00}), it is possible to endow the category of 1-motives with a tannakian structure with rational coefficients, and therefore to define the motivic Galois group  
\[\Galmot (M)\]
of a 1-motive $M$ as the fundamental group of the tannakian sub-category $< M>^\otimes$ generated by $M$ (see \cite[Def 6.1]{D89} or \cite[Def 8.13]{D90}). Applying the Generalized Grothendieck's Period Conjecture proposed by Andr\'e (see conjecture (?!) of Andr\'e's letter)
to 1-motives we get 

\begin{conjecture}[Generalized Grothendieck's Period Conjecture for 1-motives]
 Let $M$ be a 1-motive defined over a sub-field $K$ of $\CC$, then
 \begin{equation}
\mathrm{tran.deg}_{\QQ}\, K  (\mathrm{periods}(M)) \geq \dim \Galmot (M)
\label{eq:GCP}
 \end{equation}
where $K (\mathrm{periods}(M))$ is the field generated over $K$ by the periods of $M$. 
\end{conjecture}

In \cite{B02} we showed that the conjecture (\ref{eq:GCP})
applied to a 1-motive of type 
\[ M=[ u:\ZZ^{r} \,  
 \longrightarrow \,\Pi^n_{j=1} {\cE}_j \times {\GG}_m^s] \]
  is equivalent to the elliptico-toric conjecture (see \cite[1.1]{B02}) which involves elliptic integrals of the first and second kind and logarithms of complex numbers. Consider now the 1-motive  
\begin{equation}\label{eq:M}
M=[ u:\ZZ^{r} \, \longrightarrow \, G]
\end{equation}
 where $G$ is a \textit{non trivial} extension of a product $\Pi^n_{j=1} \cE_j $ of pairwise not isogenous elliptic curves by the torus $\GG_m^s.$
In this paper we introduce \textit{the 1-motivic elliptic conjecture} (\S \ref{conjecture}) which involves elliptic integrals of the first, second and third kind. Our main Theorem is that 
this 1-motivic elliptic conjecture is equivalent to the Generalized Grothendieck's Period Conjecture applied to the 1-motive (\ref{eq:M}) (Theorem \ref{thmMain}). The presence of elliptic integrals of the third kind in the 1-motivic elliptic conjecture corresponds to the fact that the extension $G$ underlying $M$ is not trivial. If in the 1-motivic elliptic conjecture we assume that the points defining the extension $G$ are trivial, then this conjecture coincides with the elliptico-toric conjecture stated in\cite[1.1]{B02} (see Remarks \ref{Rk1}). Observe that the 1-motivic elliptic conjecture contains also the Schanuel conjecture (see Remarks \ref{Rk2}).

 In Section \ref{EllipticIntegral} we recall basic facts about differential forms on elliptic curves.

 In Section \ref{periods} we study the short exact sequences which ``d\'evissent'' the Hodge and De Rham realizations of 1-motives and which are 
 induced by the weight filtration of 1-motives. In Lemma \ref{lem:decomposition} we prove that instead of working with the 1-motive (\ref{eq:M}) we can work with a direct sum of 1-motives having $r=n=s=1$. 
Using Deligne's construction of a 1-motive starting from an open singular curve,
in \cite[\S 2]{Ber08} D. Bertrand has computed
 the periods of the 1-motive (\ref{eq:M}) with $r=n=s=1.$ Putting together Lemma \ref{lem:decomposition} and  Bertrand's calculation of the periods in the case $r=n=s=1$, we compute explicitly the periods of the 1-motive (\ref{eq:M}) (see Proposition \ref{proof-periods}).

 In section \ref{motivicGaloisgroup}, which is the most technical one, we study the motivic Galois group of 1-motives.
 We will follow neither Nori and Ayoub's theories nor Grothendieck's theory involving mixed realizations, but we will work in a completely geometrical setting using \textit{algebraic geometry on tannakian categories}.
In Theorem \ref{eq:dimUR} we compute explicitly the dimension of the unipotent radical of the motivic Galois group of an arbitrary 1-motive over $K$. Then, as a corollary, we calculate explicitly the dimension of the motivic Galois group of the 1-motive (\ref{eq:M}) (see Corollary \ref{eq:dimGalMot}).
 For this last result, we restrict to work with a 1-motive whose underlying extension $G$ involves a product of elliptic curves, because only in this case we know explicitly the dimension of the reductive part of the motivic Galois group (in general, the dimension of the motivic Galois group of an abelian variety is not known). 
 
In section \ref{conjecture} we state the 1-motivic elliptic conjecture and we prove our main Theorem
 \ref{thmMain}.

In section \ref{lowDim} we compute explicitly the Generalized Grothendieck's Period Conjecture in the low dimensional case, that is assuming $r=n=s=1$ in (\ref{eq:M}). In particular we investigate the cases where $\mathrm{End}(\cE) \otimes_\ZZ \QQ$-linear dependence and torsion properties affect the dimension of the unipotent radical of $\Galmot (M)$.

\section*{Acknowledgements}

I want to express my gratitude to M. Waldschmidt for pointing out to me the study of third kind elliptic integrals and for his appendix. I am very grateful to Y. Andr\'e for his letter and for the discussions we had about the motivic Galois group. I also thank D. Bertrand and P. Philippon for their comments on an earlier version of this paper.
This paper was written during a 2 months stay at the IHES. The author thanks the Institute for the wonderful work conditions.
 

\section*{Notation}

Let $K$ be a sub-field of $\CC$ and denote by $\overline{K}$ its algebraic closure. 

A 1-motive $M=[u:X \rightarrow G]$ over $K$ consists of a
group scheme $X$ which is locally for the \'etale
topology a constant group scheme defined by a finitely generated free
$\ZZ \,$-module, an extension $G$ of an abelian variety $A$ by a torus $T$, and a homomorphism $u:X \to G(K)$. In this paper we will consider above all 1-motives in which $X= \ZZ^r$ and $G$ is an extension of a finite product $\Pi^n_{j=1} \cE_j $ of elliptic curves by the torus $\GG_m^s$ (here $r,n$ and $s$ are integers bigger or equal to 0).

There is a more symmetrical definition of 1-motives. In fact to have the 1-motive $M=[u:\ZZ^r \rightarrow G]$ is equivalent to have the 7-tuple $(\ZZ^r,\ZZ^s, \Pi^n_{j=1} \cE_j ,\Pi^n_{j=1} \cE_j^*, v ,v^*,\psi)$ where
\begin{itemize}
	\item $\ZZ^s$ is the character group of the torus $\GG_m^s$ underlying the 1-motive $M$.
	\item $v:\ZZ^r \rightarrow \Pi^n_{j=1} \cE_j$ and $v^*:\ZZ^s \rightarrow \Pi^n_{j=1} \cE_j^*$ are two morphisms of $K$-group varieties (here $\cE_j^* := \underline{\Ext}^1(\cE_j,\GG_m)$ is the Cartier dual of the elliptic curve $\cE_j$).
	To have the morphism $v$ is equivalent to have $r$ points $P_k=(P_{1k}, \ldots, P_{nk})$ of $ \Pi^n_{j=1} \cE_j(K)$ with $k=1, \ldots, r$, whereas to have the morphism $v^*$ is equivalent  to have $s$ points $Q_i=(Q_{1i}, \ldots, Q_{ni})$ of $ \Pi^n_{j=1} \cE_j^*(K)$ with $i=1, \ldots, s.$ Via the isomorphism 
	 $\underline{\Ext}^1(\Pi^n_{j=1}\cE_j,\GG_m^s) \cong  (\Pi_{j=1}^n \cE_j^*)^s ,$
	to have the $s$ points $Q_i=(Q_{1i}, \ldots, Q_{ni})$ is equivalent to have the extension $G$ of $\Pi^n_{j=1} \cE_j$ by $\GG_m^s$. 
	\item $\psi$ is a trivialization of the pull-back $(v,v^*)^*\mathcal{P}$ via $(v,v^*)$ of the Poincar\' e biextension $\mathcal{P}$ of $(\Pi^n_{j=1} \cE_j,\Pi^n_{j=1} \cE_j^*)$ by $\GG_m$. To have this trivialization $\psi$ is equivalent to have $r$ points $R_k \in G(K)$ with $k=1, \ldots, r$ such that the image of $R_k$ via the projection $G \to \Pi^n_{j=1} \cE_j$ is $P_k=(P_{1k}, \ldots, P_{nk})$, and so finally to have the morphism $u:\ZZ^r \rightarrow G. $ 
\end{itemize}

The index $k$, $0 \leq k \leq r,$ is related to the lattice $\ZZ^r$, the index $j$, $0 \leq j \leq n,$ is related to the elliptic curves, and the index $i$, $0 \leq i \leq s,$ is related to the torus $\GG_m^s$. For $j=1,  \ldots, n$, we index with a $j$ all the data related to the elliptic 
curve $\cE_j$: for example we denote by
$\wp_j(z)$ the Weierstrass $\wp$-function of $\cE_j$, by
$\omega_{j1}, \omega_{j2}$ its periods, ...

On any 1-motive $M=[u:X \rightarrow G] $ it is defined an increasing filtration  $\W_{\bullet}$, called the \textit{weight filtration} of $M$:
$\W_{0}(M)=M,  \W_{-1}(M)=[0 \to G], 
\W_{-2}(M)=[0 \to T].$
If we set ${\Gr}_{n}^{\W} :=
\W_{n} / \W_{n-1},$ we have ${\Gr}_{0}^{\W}(M)= 
[ X \to 0], {\Gr}_{-1}^{\W}(M)= 
[0 \to A]$ and $ {\Gr}_{-2}^{\W}(M)= 
[0 \to T].$

Two 1-motives $M_i=[u_i:X_i \rightarrow G_i]$ over $K$ (for $i=1,2$) are isogeneous is there exists a morphism of complexes $(f_X,f_G):M_1 \to M_2$ such that $f_X:X_1 \to X_2$ is injective with finite cokernel, and $f_G:G_1 \to G_2$ is surjective with finite kernel. 
Since \cite[Thm (10.1.3)]{D75} is true modulo isogenies, two isogeneous 1-motives have the same periods. Moreover, two isogeneous 1-motives build the same tannakian category and so they have the same motivic Galois group. Hence in this paper \textit{we can work modulo isogenies}. In particular the elliptic curves 
$\cE_1, \dots, \cE_n$ will be pairwise not isogenous.


\section{Elliptic integrals of third kind}\label{EllipticIntegral}

Let $\cE$ be an elliptic curve defined over $\CC$ with Weierstrass coordinate functions $x$ and $y$. 
Set $\Lambda := \HH_1(\cE(\CC),\ZZ). $ Let $\wp(z)$ be the Weierstrass $\wp$-function relative to the lattice $\Lambda$: it is a meromorphic function on $\CC$ having a double pole with residue zero at each point of $\Lambda$ and no other poles. 
Consider the elliptic exponential
\begin{align}
\nonumber	\exp_{\cE}: \CC & \longrightarrow  \cE(\CC) \subseteq \PP^2(\CC)\\
\nonumber	z & \longmapsto \exp_{\cE}(z)=[\wp(z),\wp(z)',1]
\end{align}
whose kernel is the lattice $\Lambda.$ In particular the map $\exp_{\cE}$
induces a complex analytic isomorphism between the quotient $\CC  / \Lambda$ and the $\CC$-valuated points of the elliptic curve $\cE$. In this paper, we will use small letters for elliptic logarithms of points on elliptic curves which are written with capital letters, that is $\exp_{\cE}(p)=P \in \cE (\CC)$ for any $p \in \CC$.

 Let $ \sigma(z)$ be the Weierstrass $\sigma$-function relative to the lattice $\Lambda$: it is a holomorphic function on all of $\CC$ and it has simple zeros at each point of $\Lambda$ and no other zeros. Finally let $\zeta (z)$ 
 be the Weierstrass $\zeta$-function relative to the lattice $\Lambda$: it is a meromorphic function on $\CC$ with simple poles at each point of $\Lambda$ and no other poles. We have the well-known equalities 
\[ \frac{d}{dz} \log \sigma(z)= \zeta(z) \quad \mathrm{and} \quad \frac{d}{dz} \zeta(z)= -\wp(z).  \]
  
Recall that a meromorphic differential 1-form is of the \emph{first kind} if it is holomorphic everywhere, of the \emph{second kind} if the residue at any pole vanishes, and of the \emph{third kind}  in general. On the elliptic curve $\cE$ we have the following differential 1-forms:
\begin{enumerate}
	\item the differential of the first kind 
	\begin{equation}\label{eq:diffFirstk}
	\omega = \frac{dx}{y}, 
	\end{equation}
	  which has neither zeros nor poles and which is invariant under translation. We have that $\exp_{\cE}^{*}(\omega) = dz.$
	\item the differential of the second kind  
	\begin{equation}\label{eq:diffSecondk}
	\eta = -\frac{xdx}{y}.
	\end{equation}
	 In particular $\exp_{\cE}^{*}(\eta) = -\wp(z) dz$ which has a double pole with residue zero at each point of $\Lambda$ and no other poles.
	\item the differential of the third kind
		\begin{equation}\label{eq:diffThirdk}
	\xi_Q = \frac{1}{2} \frac{y-y(Q)}{x - x(Q)} \frac{dx}{y}
	\end{equation}
	 for any point $Q $ of $ \cE(\CC), Q \not=0.$ The residue divisor of $\xi_Q$ is $-(0)+(-Q).$
	If we denote $q \in \CC$ an elliptic logarithm of the point $Q$, that is $\exp_{\cE}(q)=Q$, we have that 
	 \[\exp_{\cE}^{*}(\xi_Q) =  \frac{1}{2} \frac{\wp'(z)- \wp'(q)}{\wp(z) - \wp(q)} dz, \]
	  which has residue -1 at each point of $\Lambda$. 
\end{enumerate}

The 1-dimensional $\CC$-vector space of differentials of the first kind is $\HH^0(\cE, \Omega^1_\cE).$
The 1-dimensional $\CC$-vector space of differentials of the second kind modulo holomorphic differentials and exact differentials is $\HH^1(\cE, \mathcal{O}_\cE).$ In particular the first De Rham cohomology group  $\HH^1_\dR(\cE)$ of the elliptic curve $\cE$ is the direct sum $\HH^0(\cE, \Omega^1_\cE) \oplus \HH^1(\cE, \mathcal{O}_\cE)$ of these two spaces and it has dimension 2. The $\CC$-vector space of differentials of the third kind is infinite dimensional.

The inverse map of the complex analytic isomorphism $\CC / \Lambda \to \cE(\CC)$ induced by the elliptic exponential is given by the integration $\cE(\CC) \to \CC / \Lambda, P \to \int^{P}_{O} \omega \quad \mathrm{mod} \Lambda$, where O is the neutral element for the group law of the elliptic curve.

Let $\gamma_1, \gamma_2$  be two closed paths on $\cE(\CC)$ which build a basis of $\HH_1(\cE_\CC,\QQ)$. Then \textit{the elliptic integrals of the first kind} $ \int_{\gamma_i} \omega = \omega_i$ $(i=1,2)$ are \textit{the periods of the Weierstrass $\wp$-function}: 
\begin{equation}\label{eq:periods-wp}
 \wp(z+\omega_i)= \wp(z) \quad \quad  \mathrm{for} \; i=1,2.
\end{equation}
 Moreover \textit{the elliptic integrals of the second kind}
$ \int_{\gamma_i} \eta = \eta_i$ $(i=1,2)$ are \textit{the quasi-periods of the Weierstrass $\zeta$-function}:
\begin{equation}\label{eq:periods-zeta}
 \zeta(z+\omega_i)= \zeta(z) + \eta_i  \quad \quad  \mathrm{for} \; i=1,2.
 \end{equation}
Consider Serre's function 
\begin{equation}\label{eq:def-fq}
f_q(z)= \frac{\sigma(z+q)}{\sigma(z) \sigma(q)} e^{-\zeta(q) z }  \qquad \mathrm{with}\; q \in \CC \setminus \Lambda
\end{equation}
whose logarithmic differential is
\begin{equation}\label{eq:expEXiq}
\frac{f_q'(z)}{f_q(z)} dz = \frac{1}{2} \frac{\wp'(z)- \wp'(q)}{\wp(z) - \wp(q)} dz =\exp_{\cE}^{*}(\xi_Q)
\end{equation}
(see \cite{W84} and \cite[\S 2]{Ber08}). \textit{The exponentials of the elliptic integrals of the third kind}
$ \int_{\gamma_i} \xi_Q = \eta_i q - \omega_i \zeta(q)$ $(i=1,2)$ are \textit{the quasi-quasi periods} of the function $f_q(z):$
\begin{equation}\label{eq:periods-fq}
f_q(z+ \omega_i)= f_q(z) e^{\eta_i q - \omega_i \zeta(q)} \quad \quad  \mathrm{for} \; i=1,2.
\end{equation}
As observed in \cite{W84}, we have that   
\begin{equation}
\frac{f_q(z_1+ z_2)}{f_q(z_1)f_q( z_2)}= \frac{\sigma(q+z_1+z_2)\sigma(q) \sigma(z_1)\sigma(z_2)}{\sigma(q+z_1)\sigma(z_1+z_2)\sigma(q+z_2)}.
\label{eq:fq-sigma}
\end{equation}

Consider now an extension $G$ of our elliptic curve $\cE$ by $\GG_m, $ which is defined over $\CC$. Via the isomorphism $\mathrm{Pic}^0(\cE) \cong \cE^* = \underline{\Ext}^1(\cE,\GG_m)$, to have the extension $G$ 
 is equivalent to have a divisor $D=(-Q)-(0)$ of $\mathrm{Pic}^0(\cE) $ or a point $-Q$ of $ \cE^*(\CC)$. In this paper we identify $\cE$ with $\cE^*$. 
 A basis of the first De Rham cohomology group  $\HH^1_\dR(G)$ of the extension $G$ is given by $\{\omega, \eta, \xi_Q \}$. Consider the semi-abelian exponential
\begin{equation}\label{eq:semiablog}
	\exp_{G}: \CC^2  \longrightarrow  G(\CC) \subseteq \PP^4(\CC)
	\end{equation}
\[	(w,z) \longmapsto \exp_{G}(w,z)=\sigma(z)^3 \Big[\wp(z),\wp(z)',1, e^{w} f_q(z), e^{w} f_q(z) \Big( \wp(z) + \frac{\wp'(z)- \wp'(q)}{\wp(z)- \wp(q)} \Big) \Big]
\]
 whose kernel is $\HH_1(G(\CC),\ZZ)$. A basis of the Hodge realization  $\HH_1(G(\CC),\QQ)$ of the extension $G$ is given by 
  a closed path $\delta_{Q}$ around $Q$ on $G(\CC)$ and 
  two closed paths $\tilde{\gamma}_1, \tilde{\gamma}_2$ on $G(\CC)$ which lift a basis  $\{\gamma_1, \gamma_2\}$ of $\HH_1(\cE_\CC,\QQ)$ via the surjection $ \HH_1(G_\CC,\QQ) \rightarrow \HH_1(\cE_\CC,\QQ).$
 We have that  
  \begin{equation}\label{eq:expGXiq}
 \exp_{G}^{*}(\xi_Q) = dw + \frac{f_q'(z)}{f_q(z)} dz.
  \end{equation}


\section{Periods of 1-motives involving elleptic curves}\label{periods}

Let $M=[u:X \to G]$ be a 1-motive over $K$ with $G$ an extension of an abelian variety $A$ by a torus $T$. 
As recalled in the introduction, to the 1-motive $M_{\CC}$ obtained from $M$ extending the scalars from $K$ to $\CC$,
we can associate its Hodge realization ${\T}_{\QQ}(M_\CC)= (\Lie(G_\CC)\times_G X) \otimes_\ZZ \QQ $ which is endowed with the weight filtration (defined over the integers) ${\W}_{0}{\T}_{\ZZ}(M_\CC) =\Lie(G_\CC)\times_G X, {\W}_{-1}{\T}_{\ZZ}(M_\CC) = {\HH}_1(G_\CC,\ZZ),  {\W}_{-2}{\T}_{\ZZ}(M_\CC) = {\HH}_1( T_\CC,\ZZ).$
 In particular we have that 
 ${\Gr}_0^{\W}{\T}_{\ZZ}(M_\CC)\cong X,   {\Gr}_{-1}^{\W}{\T}_{\ZZ}(M_\CC)
 \cong {\HH}_{1}(A_\CC,\ZZ)$ and
  $ {\Gr}_{-2}^{\W}{\T}_{\ZZ}(M_\CC) \cong {\HH}_{1}(T_\CC,\ZZ).$
Moreover to $M$ we can associate its De Rham realization ${\T}_{\dR}(M) = \Lie (G^\natural)$, where $M^\natural=[X \rightarrow G^\natural]$ is the universal vectorial extension of $M$, which is endowed with the Hodge filtration ${\F}^0{\T}_{\dR}(M)= \ker \big( \Lie ( G^\natural) \rightarrow \Lie ( G) \big).$

The weight filtration induces for the Hodge realization the short exact sequence 
\begin{equation}\label{eq:Hodge}
0 \longrightarrow \HH_1(G_\CC,\ZZ)  \longrightarrow \T_{\ZZ}  ( M_\CC) \longrightarrow \T_{\ZZ}  (X) \longrightarrow 0 
\end{equation}
which is not split in general. On the other hand, for the De Rham realization we have that

\begin{lemma}
 The short exact sequence, induced by the weight filtration,
\begin{equation}\label{eq:DRham0}
 0 \longrightarrow {\T}_{\dR}(G)  \longrightarrow {\T}_{\dR}(M) \longrightarrow {\T}_{\dR}(X) \longrightarrow 0 
\end{equation}
is canonically split.
\end{lemma}

\begin{proof}
Consider the short exact sequence $0 \to G \to M \to X[1] \to 0$. Applying $\Hom(-,\GG_a)$ we get the short exact sequence of finitely dimensional $K$-vector spaces
\[ 0 \longrightarrow \Hom (X,\GG_a) \longrightarrow \Ext^1(M,\GG_a) \to\Ext^1(G,\GG_a) \longrightarrow 0 \]
Taking the dual we obtain the short exact sequence
\[ 0 \longrightarrow \Hom (\Ext^1(G,\GG_a),\GG_a) \longrightarrow \Hom (\Ext^1(M,\GG_a),\GG_a) \longrightarrow X  \to 0 \]
which is split since $\Ext^1(X, \GG_a)=0$.
Now consider the composite of the section 
\pn $X \to \Hom (\Ext^1(M,\GG_a),\GG_a)$ with the inclusion 
$\Hom (\Ext^1(M,\GG_a),\GG_a) \to G^\natural$.
Recalling that  ${\F}^0{\T}_{\dR}(M) \cong \Hom (\Ext^1(M,\GG_a),\GG_a)$, if we 
take Lie algebras we get
the arrow $\T_{\dR} (X) = X \otimes K \to {\F}^0{\T}_{\dR}(M) \to  {\T}_{\dR}(M) = \Lie (G^\natural)$ which is a section of the exact sequence (\ref{eq:DRham0}).
\end{proof}

Denote by $\HH_{\dR}(M) $ the dual $K$-vector space  of ${\T}_{\dR}(M)$. By the above Lemma we have that
\begin{equation}\label{eq:DRham}
\HH_{\dR}(M) = \HH_{\dR}^1(G) \oplus \HH_{\dR}^1(X) .
\end{equation}

Consider now a 1-motive $M=[u:\ZZ^r \rightarrow G]$ defined over $K$, 
where $G$ is an extension of a finite product $\Pi^n_{j=1} \cE_j $ of elliptic curves by the torus $\GG_m^s$. Let $\{ z_k \}_{k=1, \dots, r}$ be a basis of $\ZZ^r$ and let  $\{ t_i \}_{i=1, \dots, s}$ be a basis of the character group $\ZZ^s$ of $\GG_m^s$. For the moment, in order to simplify notation, denote by $A$ the product of elliptic curves $\Pi^n_{j=1} \cE_j$.
Denote by $G_i$ the push-out of G by $t_i: \GG_m^s \to \GG_m$, which is the extension of $A$ by $\GG_m$ parameterized by the point $v^*(t_i)=Q_i=(Q_{1i}, \dots, Q_{ni})$, and by $R_{ik}$ the $K$-rational point of $G_i$ 
above $v(z_k)=P_k=(P_{1k}, \dots, P_{nk})$.
Consider the 1-motive defined over $K$
\[M_{ik}= [u_{ik}:z_k \ZZ \rightarrow G_i]\]
 with $u_{ik}(z_k)= R_{ik} $ for $i=1, \dots, s$ and $k=1, \dots, r$.
In \cite[Thm 1.7]{B02-2} we have proved geometrically that the 1-motives $M=[u:\ZZ^r \rightarrow G]$ and $\oplus_{i=1}^s \oplus_{k=1}^r M_{ik}$    generate the same tannakian category.  
 Via the isomorphism $\underline{\Ext}^1(\Pi^n_{j=1}\cE_j,\GG_m) \cong  \Pi_{j=1}^n \underline{\Ext}^1(\cE_j,\GG_m) ,$ 
 the extension $G_i$ of $A$ by $\GG_m$ parametrized by the point $v^*(t_i)=Q_i=(Q_{1i}, \dots, Q_{ni})$ corresponds to the product of extensions $G_{1i} \times G_{2i}  \times \dots \times  G_{ni}$ where 
 $ G_{ji}$ is an extension of $\cE_j$ by $\GG_m$ parametrized by the point $Q_{ji}$, and the $K$-rational point $R_{ik}$ of $G_i$ living above $P_k=(P_{1k}, \dots, P_{nk})$ corresponds to the 
 $K$-rational points $(R_{1ik}, \dots, R_{nik})$ with $R_{jik} \in  G_{ji}(K) $ living above $P_{jk} \in \cE_j (K).$
 Consider the 1-motive defined over $K$
 \begin{equation}\label{eq:jik}
M_{jik}= [u_{jik}:z_k \ZZ \rightarrow G_{ji}]
\end{equation}
with $u_{jik}(z_k)= R_{jik} $ for $i=1, \dots, s$, $k=1, \dots, r$ and $j=1, \dots, n.$  Let $(l_{jik},p_{jk}) \in \CC^2$ be a semi-abelian logarithm (\ref{eq:semiablog}) of $R_{jik},$ that is 
\begin{equation}\label{eq:l}
\exp_{G_{ji}} (l_{jik},p_{jk}) = R_{jik}.
\end{equation}

 \begin{lemma}\label{lem:decomposition}
 	The 1-motives $M$ and $\oplus_{i=1}^s \oplus_{k=1}^r \oplus_{j=1}^n M_{jik}$ generate the same tannakian category.  
 	\end{lemma}
 	
 	\begin{proof} As in \cite[Thm 1.7]{B02-2} we will work geometrically and because of loc. cit. it is enough to show that the 1-motives 
 		 $\oplus_{i=1}^s \oplus_{k=1}^r M_{ik}$  and  $\oplus_{i=1}^s \oplus_{k=1}^r \oplus_{j=1}^n M_{jik}$ generate the same tannakian category. Clearly
 		\[ \oplus_{j=1}^n  \Big( \oplus_{i=1}^s \oplus_{k=1}^r 
 		 M_{ik} \big/ [0 \to \Pi_{1 \leqslant l \leqslant n \atop l \not= j} G_{li} ]
 		\Big)   
 		= \oplus_{i=1}^s \oplus_{k=1}^r \oplus_{j=1}^n M_{jik} \]
 and so $< \oplus_{i=1}^s \oplus_{k=1}^r \oplus_{j=1}^n M_{jik}>^\otimes \; \;  \subset \;  \;  < \oplus_{i=1}^s \oplus_{k=1}^r M_{ik}>^\otimes.$
On the other hand, if $\mathrm{d}_\ZZ: \ZZ \to \ZZ^n$ is the diagonal morphism, for fixed $i$ and $k$ we have that
\[  \oplus_{j=1}^n M_{jik} \big/ [ \ZZ^n / \mathrm{d}_\ZZ(\ZZ) \to 0] =
[\small{\Pi}_{j} u_{jik} : \mathrm{d}_\ZZ(\ZZ) \longrightarrow G_{1i} \times G_{2i}  \times \dots \times  G_{ni}] = [u_{ik} :\ZZ  \longrightarrow G_i] =M_{ik}
 \] 
 and so 
 \[\oplus_{i=1}^s \oplus_{k=1}^r  \Big(  \oplus_{j=1}^n M_{jik} \big/ [ \ZZ^n / \mathrm{d}_\ZZ(\ZZ) \to 0] \Big) = \oplus_{i=1}^s \oplus_{k=1}^r M_{ik}   \]
 that is $ < \oplus_{i=1}^s \oplus_{k=1}^r M_{ik}>^\otimes \; \;  \subset \;  \; < \oplus_{i=1}^s \oplus_{k=1}^r \oplus_{j=1}^n M_{jik}>^\otimes .$ 
	\end{proof}

  The matrix which represents the isomorphism (\ref{eq:betaM}) for the 1-motive $M=[u:{\ZZ}^r \to G]$, where $G$ is an extension of $\Pi^n_{j=1} \cE_j $  by  $\GG_m^s$, is a huge matrix difficult to write down. The above Lemma implies that, instead of studying this huge matrix, it is enough to study the $rsn$ matrices which represent the isomorphism (\ref{eq:betaM}) for the $rsn$ 1-motives $M_{jik}= [u_{jik}:z_k \ZZ \rightarrow G_{ji}].$

  Following \cite[\S 2]{Ber08}, now we compute explicitly the periods of the  1-motive $M=[u:\ZZ \to G]$, where $G$ is an extension of one elliptic curve $ \cE$ by the torus $\GG_m.$ We need Deligne's 
  construction of $M$ starting from an open singular curve (see \cite[(10.3.1)-(10.3.2)-(10.3.3]{D75}) that we recall briefly.
   Via the isomorphism $\mathrm{Pic}^0(\cE) \cong \cE^* = \underline{\Ext}^1(\cE,\GG_m)$,
  to have the extension $G$ of $\cE$ by $\GG_m$ underlying the 1-motive $M$ is equivalent to have the divisor $D=(-Q)-(0)$ of $\mathrm{Pic}^0(\cE)$ or the point $-Q$ of $ \cong \cE^* $. We assume $Q$ to be a non torsion point.
   According to \cite[page 227]{M74}, to have the point $u(1)=R \in G(K)$ is equivalent to have a couple 
   \[(P,g_R) \in \cE(K) \times K(\cE)^*\]
    where  $\pi(R)=P \in \cE(K)$ (here $\pi: G \to \cE$ is the surjective morphism of group varieties underlying the extension $G$),
    and where $g_R:  \cE \to \GG_m, x \mapsto R+ \rho(x) -\rho(x+P)$  (here $\rho: \cE \to G$ is a section of $\pi$), is a rational function on $\cE$  whose divisor is $T^{*}_{P}D-D=(-Q+P)-(P)-(-Q)+(0)$ (here $T_P: \cE \to \cE$ is the translation by the point $P$).   We assume also $R$ to be a non torsion point.

  Now pinch the elliptic curve $\cE$ at the two points $-Q$ and $O$ and puncture it at
  two $K$-rational points $P_2$ and $P_1$ whose difference (according to the group law of $\cE$) is $P$, that is $P=P_2-P_1.$ The motivic $\HH^1$ of the open singular curve obtained in this way from $\cE$ is the 1-motive $M=[u:\ZZ \rightarrow G]$, with $u(1)=R$. We will apply Deligne's construction to each 1-motive $M_{jik}= [u_{jik}:z_k \ZZ \rightarrow G_{ji}]$ with $u_{jik}(z_k)= R_{jik} .$

 \begin{proposition}\label{proof-periods}
 Choose the following basis of the $\QQ$-vector space  $\T_{\QQ}(M_{jik \; \CC}):$
 	\begin{itemize}
	\item two closed paths $\tilde{\gamma}_{j1}, \tilde{\gamma}_{j2}$  on $G_{ji}(\CC)$ which lift the basis  $\{\gamma_{j1}, \gamma_{j2}\}$ of $\HH_1(\cE_{j \;\CC},\QQ)$ via the surjection $ \HH_1(G_{ji \; \CC},\QQ) \rightarrow \HH_1(\cE_{j \;\CC},\QQ)$; 
 		\item a closed path $\delta_{Q_{ji} }$ around $-Q_{ji}$ on $G_{ji}(\CC)$ (here we identify $G_{ji}$ with the pinched elliptic curve $\cE_j$); and
 		\item a closed path $\beta_{R_{jik}}$, which lifts the basis $\{z_k\}$ of $\T_{\QQ}(z_k \ZZ)$  via the surjection $ \T_{\QQ}  ( M_{jik \; \CC}) \rightarrow \T_{\QQ}  (z_k \ZZ) $, and whose restriction to $\HH_1(G_{ji \; \CC},\QQ)$ is a closed path $\beta_{R_{jik}|G_{ji}}$ on 
 		$G_{ji}(\CC)$ having the following properties: $\beta_{R_{jik}|G_{ji}}$ lifts a path 
 		$\beta_{P^1_{jk}P^2_{jk}}$ on $\cE_{j }(\CC) $ from $P^1_{jk}$ to $P^2_{jk}$ (with $P^2_{jk}-P^1_{jk}=P_{jk}$) via the surjection $ \HH_1(G_{ji \; \CC},\QQ) \rightarrow \HH_1(\cE_{j \;\CC},\QQ)$, and its restriction to $\HH_1(\GG_m,\QQ)$ 
 		is a path $\beta_{jik}$ on $\GG_m(\CC)=\CC^* $ from $1$ to $l_{jik} (\ref{eq:l});$
 \end{itemize}	
 	and the following basis of the $K$-vector space $\HH_{\dR}(M_{jik}):$	
 	\begin{itemize}	
        \item  the differentials of the first kind $\omega_j=\frac{dx_j}{y_j}$ (\ref{eq:diffFirstk}) and of the second kind $\eta_j=-\frac{x_jdx_j}{y_j}$ (\ref{eq:diffSecondk}) of $\cE_j$;
 		\item the differential of the third kind $\xi_{Q_{ji} }=
 		 \frac{1}{2} \frac{y_j-y_j(Q_{ji})}{x_j - x_j(Q_{ji})} \frac{dx_i}{y_j}$ (\ref{eq:diffThirdk}) of $\cE_j$, whose residue divisor is $D=(-Q_{ji})-(0)$
 		and which lifts the basis $\{\frac{dt_i}{t_i}\}$ of  $\HH_{\dR}^1(\GG_m)$ via the surjection $\HH_{\dR}^1(G_{ji}) \rightarrow \HH_{\dR}^1(\GG_m)$;
 		\item the differential $df_j$ of a rational function $f_j$ on $\cE_j$ such that $f_j(P^2_{jk})$ differs from $f_j(P^1_{jk})$ by 1.		
 	\end{itemize}
 	 These periods of the 1-motive $M=[u:{\ZZ}^r \to G]$, where $G$ is an extension of $\Pi^n_{j=1} \cE_j $  by  $\GG_m^s$, are then
 	\[1,  \omega_{j1},\omega_{j2},\eta_{j1}, \eta_{j2}, p_{jk},\zeta_j(p_{jk}),
 	\eta_{j1} q_{ji} - \omega_{j1} \zeta_j(q_{ji}),\eta_{j2} q_{ji} - \omega_{j2} \zeta_j(q_{ji}) , 
 	\log f_{q_{ji}}(p_{jk}) + l_{jik}, 2i \pi \] 
 with $e^{l_{jik}} \in K^*,$	for $j=1,  \ldots, n, k=1,  \ldots, r$ and $i=1,  \ldots,s.$
 \end{proposition}

\begin{proof}
By Lemma \ref{lem:decomposition}, the 1-motives $M=[u:\ZZ^r \rightarrow G]$  and $\oplus_{i=1}^s \oplus_{k=1}^r \oplus_{j=1}^n  [u_{jik}:z_k \ZZ \rightarrow G_{ji}]$ have the same periods and therefore
 we are reduced to prove the case $r=n=s=1$. 
 
 Consider the 1-motive $M=[u: z\ZZ \to G],$ where $G$ is an extension of an elliptic curve $\cE$ by $\GG_m$ parameterized by $v^*(t)=-Q \in \cE(K)$, and $u(z)=R$ is a point of $G(K)$ living over $v(z)=P \in \cE(K).$ Let $(l,p) \in \CC^2$ be a semi-abelian logarithm of $R,$ that is 
 \[\exp_G (l,p) = R. \] 
  Let $P_2$ and $P_1$ be two $K$-rational points whose difference is $P$.
 Because of the weight filtration of $M$, we have the non-split short exact sequence 
\[
0 \longrightarrow \HH_{\dR}^1(\cE)  \longrightarrow \HH_{\dR}^1(G) \longrightarrow \HH_{\dR}^1(\GG_m) \longrightarrow 0 
\]
As $K$-basis of $\HH_{\dR}^1(G)$ we choose the differentials of the first kind $\omega $ and of the second kind $\eta$ of $\cE,$ and the differential of the third kind $\xi_Q$, which lifts the only element $\frac{dt}{t}$ of the basis of $\HH_{\dR}^1(\GG_m)$. Because of the decomposition (\ref{eq:DRham}), we complete the basis of $\HH_{\dR}(M)$ with the differential $df$ of a rational function $f$ on $\cE$ such that $f(P_2)$ differs from $f(P_1)$ by 1.
\par\noindent Always because of the weight filtration of $M$, we have the non-split short exact sequence 
\[
0 \longrightarrow \HH_1(\GG_m,\ZZ)  \longrightarrow \HH_1(G_\CC,\ZZ) \longrightarrow \HH_1(\cE_\CC,\ZZ) \longrightarrow 0 
\]
As $\QQ$-basis of $ \HH_1(G_\CC,\QQ)$ we choose two closed paths $\tilde{\gamma}_1, \tilde{\gamma}_2$ which lift the basis $\gamma_1, \gamma_2$ of $\HH_1(\cE_\CC,\QQ)$ and a closed path $\delta_{Q}$ around $-Q$.
Because of the non-split exact sequence (\ref{eq:Hodge}), we complete the basis of $\T_{\QQ}(M)$ with 
a closed path $\beta_R$, which lifts the only element $z$ of the basis of $\T_{\QQ}(z \ZZ) = \ZZ \otimes \QQ$  via the surjection $ \T_{\QQ}  ( M_{ \CC}) \rightarrow \T_{\QQ}  (z \ZZ) $, and whose restriction 
to $\HH_1(G_{ \CC},\QQ)$ is a closed path $\beta_{R|G}$ on 
$G(\CC)$ having the following properties: $\beta_{R|G}$ lifts a path 
$\beta_{P_1P_2}$ on $\cE(\CC) $ from $P_1$ to $P_2$, and its restriction to $\HH_1(\GG_m,\QQ)$ 
is a path $\beta_l$ on $\GG_m(\CC)=\CC^* $ from $1$ to $l.$
With respect to these bases of $\T_{\QQ}(M)$ and  $\HH_{\dR}(M)$,
 the matrix which represents the isomorphism (\ref{eq:betaM}) for the 1-motive $M=[u: z \ZZ \to G]$ is
\begin{equation}\label{eq:matrix-integrales}
\left( {\begin{array}{cccc}
	\int_{\beta_R} df &\int_{\beta_{P_1P_2}} \omega & \int_{\beta_{P_1P_2}} \eta &\int_{\beta_{R|G}}  \xi_Q \\
	\int_{\tilde{\gamma}_1}df &\int_{\gamma_1} \omega & \int_{\gamma_1} \eta &\int_{\tilde{\gamma}_1} \xi_Q \\
	\int_{\tilde{\gamma}_2}df &\int_{\gamma_2} \omega & \int_{\gamma_2} \eta &\int_{\tilde{\gamma}_2} \xi_Q \\
	\int_{\delta_{Q}}df &\int_{\delta_{Q}}\omega & \int_{\delta_{Q}}\eta & \int_{\delta_{Q}} \xi_Q \\
	\end{array} } \right)
\end{equation}

Recalling that $\exp_{\cE}^{*}(\omega) = dz, \exp_{\cE}^{*}(\eta) =  d \zeta(z)$, (\ref{eq:expEXiq})  and (\ref{eq:expGXiq}) we can now compute explicitly all these integrals:
\begin{itemize}
		\item $\int_{\beta_R} df= f(P_2)-f(P_1)=1,$
	\item $\int_{\tilde{\gamma}_1}df=\int_{\tilde{\gamma}_2}df = \int_{\delta_{Q}}df =0$ because of the decomposition (\ref{eq:DRham}),
	\item $\int_{\beta_{P_1P_2}} \omega= \int_{p_1}^{p_2} dz= p_2 - p_1 =p,$
	\item $\int_{\gamma_i} \omega= \int_{0}^{\omega_i} dz=    \omega_i$ for $i=1,2,$
		\item  $\int_{\delta_{Q}}\omega =
	\int_{\delta_{Q}}\eta=
	0$ since the image of $\delta_{Q}$ via  $\HH_1(G_\CC,\QQ) \to \HH_1(\cE_\CC,\QQ)$ is zero,
		\item $\int_{\gamma_i} \eta = \int_{0}^{\omega_i} d \zeta= \zeta(\omega_i) - \zeta(0) =\eta_i$ for $i=1,2,$
	\item $\int_{\beta_{P_1P_2}} \eta= 
  \int_{p_1}^{p_2} d \zeta(z) = \zeta(p_2) - \zeta(p_1). $
\end{itemize}
By the pseudo addition formula for the Weierstrass $\zeta$-function (see \cite[Example 2, p 451]{WW}), $\zeta(z+y) - \zeta(z)- \zeta(y) = \frac{1}{2} \frac{\wp'(z)-\wp'(y)}{\wp(z)-\wp(y)} \in K(\cE)$, and so it exists a rational function $g$ on $\cE$ such that $g(p_2)-g(p_1)= - \zeta(p + p_1) + \zeta(p)+ \zeta(p_1).$
Since the differential of the second kind $\eta$ lives in the quotient space $\HH^1(\cE, \mathcal{O}_\cE),$
we can add to the class of $\eta $ the exact differential $dg$, getting
\begin{itemize}
	\item $\int_{\beta_{P_1P_2}}( \eta + dg) = \int_{p_1}^{p_2}( d \zeta(z) + dg) =  \zeta(p_2) - \zeta(p_1) +g(p_2) -g(p_1) = \zeta(p),$
	\item $\int_{\beta_{R|G}}  \xi_Q= \int_0^l dw +\int_{p_1}^{p_2} \frac{f_q'(z)}{f_q(z)} dz = l+
	\int_{p_1}^{p_2} d \log f_q(z) = l + \log \frac{f_q(p_2)}{f_q(p_1)} .$
\end{itemize}
Since by \cite[20-53]{WW} the quotient of $\sigma$-functions is a rational function on $\cE$, 
from the equality (\ref{eq:fq-sigma}) it exists a rational function $g_q(z)$ on $\cE$ such that 
$ \frac{g_q(p_2)}{g_q(p_1)} = (\frac{f_q(p+p_1)}{f_q(p)f_q(p_1)})^{-1}$, and therefore we get
\begin{itemize}
	\item $\int_{\beta_{R|G}} ( \xi_Q + d \log g_q(z)) =  \int_0^l dw+ \int_{p_1}^{p_2}( d \log f_q(z) + d \log g_q(z)) = l+ \log \big( \frac{f_q(p_2)}{f_q(p_1)} 
	\frac{g_q(p_2)}{g_q(p_1)} \big) = \\
l+	\log \big( \frac{f_q(p_2)}{f_q(p_1)} \frac{f_q(p)f_q(p_1)}{f_q(p_2)} \big) = l+ \log( f_q(p) ),$ with $e^l \in K^*$,
	\item $\int_{\tilde{\gamma}_i} \xi_Q = \int_{0}^{\omega_i} \frac{f_q'(z)}{f_q(z)} dz = \int_{0}^{\omega_i} d \log f_q(z) =  \log \frac{f_q(\omega_i)}{f_q(0)} =\eta_i q - \omega_i \zeta(q) $ by (\ref{eq:periods-fq}) for $i=1,2,$
	\item $\int_{\delta_{-Q}} \xi_Q = 2i \pi \mathrm{Res}_{-Q}  \xi_Q = 2 i \pi.$
\end{itemize}
The addition of the differential $d \log g_q(z)$ to the differential of the third kind $\xi_Q$ will modify the last two integrals by an integral multiple of $2 i \pi$ (see \cite[Thm 10-7]{S}) and this is irrelevant for the computation of the field generated by the periods of $M.$

Explicitly the matrix (\ref{eq:matrix-integrales}) becomes
\begin{equation}\label{eq:matrix-periods}
\left( {\begin{array}{cccc}
	1 &p & \zeta(p) &\log f_q(p) + l \\
	0 & \omega_1 &  \eta_1 &\eta_1 q - \omega_1 \zeta(q)  \\
	0 & \omega_2 & \eta_2 &\eta_2 q - \omega_2 \zeta(q)  \\
	0 &0 & 0 & 2 i \pi\\
	\end{array} } \right),
\end{equation}
with $e^l \in K^*,$ and so the periods of the 1-motive $M=[u:z \ZZ \to G], u(z)=R,$ are $1, \omega_1,\omega_2, \eta_1, \eta_2, p, \zeta(p) , \log f_q(p) + l , \eta_1 q - \omega_1 \zeta(q), \eta_2 q - \omega_2 \zeta(q), 2 i \pi.$
\end{proof}

\begin{remark} The determinations of the complex and elliptic logarithms, 
	which appear in the first line of the matrix (\ref{eq:matrix-periods}), are not well-defined since they depend on the lifting $\beta_{P_1P_2}$ of
	 the basis of $\T_{\QQ}(z\ZZ)$ (recall that the short exact sequence (\ref{eq:Hodge}) is not split).
	Nevertheless, the field  $K (\mathrm{periodes}(M))$, which is involved in the Generalized Grothendieck's Period Conjecture, is totally independent of these choices since it contains $2i \pi$, the periods of the Weierstrass $\wp$-function, the quasi-periods of the Weierstrass $\zeta$-function, and finally the quasi-quasi-periods of Serre's function $f_q(z)$ (\ref{eq:def-fq}).
\end{remark}

 We finish this section with an example: Consider the 1-motive $M=[u:{\ZZ}^2 \to G]$, where $G$ is an extension of $\cE_1 \times \cE_2 $  by  $\GG_m^3$ parameterized by the $K$-rational points $Q_1=(Q_{11},Q_{21}), 
Q_2=(Q_{12},Q_{22}), Q_3=(Q_{13},Q_{23})$ of $\cE_1^* \times \cE_2^* $, and the morphism $u$ corresponds to two $K$-rational points $R_1,R_2$ of $G$ leaving over two points $P_1=(P_{11},P_{21}), 
P_2=(P_{12},P_{22})$ of $\cE_1 \times \cE_2. $ The more compact way to write down the matrix  which represents the isomorphism (\ref{eq:betaM}) for our  1-motive $M=[u:{\ZZ}^2 \to G]$ is to consider the 1-motive 
\[ M'= M/ [0 \longrightarrow \cE_1] \oplus  M/ [0 \longrightarrow \cE_2], \]
that is, with the above notation $M'=[u_1={\ZZ}^2 \to \Pi_{i=1}^3 G_{1i} ] \oplus [u_2={\ZZ}^2 \to \Pi_{i=1}^3 G_{2i} ] $ with $u_1$ corresponding to 
two $K$-rational points $(R_{111},R_{121},R_{131} )$ and $(R_{112},R_{122},R_{132})$  of $\Pi_{i=1}^3 G_{1i}$ living over $P_{11}$ and 
$P_{12}$, and $u_2$ corresponding to 
two $K$-rational points $(R_{211},R_{221},R_{231} )$ and $(R_{212},R_{222},R_{232})$ of $\Pi_{i=1}^3 G_{2i}$ living over $P_{21}$ and 
$P_{22}$.
  The 1-motives $M$ and $M'$ generate the same tannakian category: in fact, it is clear that 
$ <M'>^\otimes \; \;  \subset \;  \; < M>^\otimes $ and in the other hand
$M= M' / [ \ZZ^2 / \mathrm{d}_\ZZ(\ZZ) \to 0]$.
The matrix representing the isomorphism (\ref{eq:betaM}) for the 1-motive $M'$ with respect to the $K$-bases chosen in the above Corollary is

$$
\left(\begin{matrix} 
    &  &\scriptstyle{ p_{11} } &\scriptstyle{ \zeta_1(p_{11})} &    
& 0 &  0 &
\scriptstyle{\log f_{q_{11}}(p_{11})+l_{111}}&  \scriptstyle{\log f_{q_{12}}(p_{11})+l_{121}} & \scriptstyle{\log f_{q_{13}}(p_{11})+l_{131}}  \cr
   \scriptstyle{ {\rm Id}_{4 \times    4}  } &   &   \scriptstyle{ p_{12} } & \scriptstyle{ \zeta_1(p_{12})} 
&  & 0 &0 & \scriptstyle{\log f_{q_{11}}(p_{12})+l_{112}} &  \scriptstyle{\log f_{q_{12}}(p_{12})+l_{121}} &\scriptstyle{\log f_{q_{13}}(p_{12})+l_{131}}   \cr
   &    &0 &0 &    
&\scriptstyle{ p_{21} } & \scriptstyle{ \zeta_2(p_{21})} &
\scriptstyle{\log f_{q_{21}}(p_{21})+l_{211}}&  \scriptstyle{\log f_{q_{22}}(p_{21})+l_{221}} & \scriptstyle{\log f_{q_{23}}(p_{21})+l_{231}}  \cr
 &  &  0 & 0
&  & \scriptstyle{ p_{22} } &\scriptstyle{ \zeta_2(p_{22})}  & \scriptstyle{\log f_{q_{21}}(p_{22})+l_{212}} &  \scriptstyle{\log f_{q_{22}}(p_{22})+l_{222}} &\scriptstyle{\log f_{q_{23}}(p_{22})+l_{232}}   \cr
 &   & \scriptstyle{ { \omega}_{11}} &\scriptstyle{ { \eta}_{11}}& &    &  &  \scriptstyle{ { \eta}_{11} q_{11}-{ \omega}_{11} \zeta_1(q_{11})  }& 
\scriptstyle{ { \eta}_{11} q_{12}-{ \omega}_{11} \zeta_1(q_{12})  }&  \scriptstyle{ { \eta}_{11} q_{13}-{ \omega}_{11} \zeta_1(q_{13})  }\cr
    & & \scriptstyle{ { \omega}_{12}} &\scriptstyle{ { \eta}_{12}}& &   &  & \scriptstyle{ { \eta}_{12} q_{11}-{ \omega}_{12} \zeta_1(q_{11})}& 
\scriptstyle{ { \eta}_{12} q_{12}-{ \omega}_{12} \zeta_1(q_{12})}& \scriptstyle{ { \eta}_{12} q_{13}-{ \omega}_{12} \zeta_1(q_{13})} \cr
   &   & &  & &\scriptstyle{ { \omega}_{21}}  & \scriptstyle{ { \eta}_{21} } & \scriptstyle{ { \eta}_{21} q_{21}-{ \omega}_{21} \zeta_2(q_{21})}& 
\scriptstyle{ { \eta}_{21} q_{22}-{ \omega}_{21} \zeta_2(q_{22})}& \scriptstyle{ { \eta}_{21} q_{23}-{ \omega}_{21} \zeta_2(q_{23})} \cr
  &  & &  &  &\scriptstyle{ { \omega}_{22}}  & \scriptstyle{ { \eta}_{22}}&
\scriptstyle{ { \eta}_{22} q_{21}-{ \omega}_{22} \zeta_2(q_{21})}
& 
\scriptstyle{ { \eta}_{22} q_{22}-{ \omega}_{22} \zeta_2(q_{22})} &  \scriptstyle{ { \eta}_{22} q_{23}-{ \omega}_{22} \zeta_2(q_{23})} \cr
   &   &   & & &  &   & &  \scriptstyle{ 2i \pi {\rm Id}_{3 \times    3}  }&   
\end{matrix}	\right).
$$

In general, for a 1-motive of the kind $M=[u:\ZZ^r \rightarrow G]$,
where $G$ is an extension of a finite product $\Pi^n_{j=1} \cE_j $ of elliptic curves by the torus $\GG_m^s$, we will
consider the 1-motive
\[ M'= \oplus_{j=1}^n  \big( M/ [0 \longrightarrow \Pi_{1 \leq l \leq n \atop l \not =j }\cE_j] \big)  \]
whose matrix representing the isomorphism (\ref{eq:betaM}) with respect to the $K$-bases chosen in the above Corollary is
 $$
 \left(\begin{matrix} 
A&B&C \cr
0 & D&E \cr
0 &0 & F
\end{matrix}	\right)
$$
with $A=  {\rm Id}_{rn \times    rn},B$ the $rn \times 2n$ matrix involving the periods coming from the morphism $v: \ZZ^r \to \Pi^n_{j=1} \cE_j $ , $C$ the $rn \times s$  matrix involving the periods coming from the trivialization $\Psi$ of the pull-back via $(v,v^*)$ of the Poincar\'e biextension $\mathcal{P}$ of $(\Pi^n_{j=1}\cE_j, \Pi^n_{j=1}\cE_j^*)$ by $\GG_m$ , $D$ the $2n \times 2n$ matrix having in the diagonal the period matrix of each elliptic curves $\cE_j$, $ E$ the $2n \times s$ matrix involving the periods coming from the morphism $v^*: \ZZ^s \to \Pi^n_{j=1} \cE_j^* $, and finally $ F =  2i \pi{\rm Id}_{s \times  s}$ the period matrix of $\GG_m^s$.


\section{Dimension of the unipotent radical of the motivic Galois group of a 1-motive}\label{motivicGaloisgroup}

Denote by $\mathcal{MM}_{\leq 1}(K)$ the category of 1-motives defined over $K$. Using Nori's and Ayoub's works (see \cite{Ay14} and \cite{N00}), it is possible to endow the category of 1-motives with a \textit{tannakian structure with rational coefficients} (roughly speaking a tannakian category $\mathcal{T}$ with rational coefficients is an abelian category with a functor $\otimes :\mathcal{T} \times \mathcal{T} \to \mathcal{T}$ defining the tensor product of two objects  of $\mathcal{T} $, and with a fibre functor over $\mathrm{Spec}(\QQ)$ - see  \cite[2.1, 1.9, 2.8]{D90} for details). We work in a completely  geometrical setting using algebraic geometry on tannakian category and defining 	
as one goes along the objects, the morphisms and the tensor products that we will need (essentially we tensorize motives with pure motives of weight 0, and as morphisms we use projections and biextensions).

The unit object of the tannakian category $\mathcal{MM}_{\leq 1}(K)$ is the 1-motive $\ZZ(0)= [ \ZZ \to 0]$. In this section we use the notation $Y(1)$ for the torus whose cocharacter group is $Y$. In particular $\ZZ(1)= [ 0 \to \GG_m]$. If $M$ is a 1-motive, we denote by $M^\vee  \cong \uHom (M, \ZZ(0))$ its dual and by $ev_M : M \otimes M^\vee \to \ZZ(0), \delta_M:  \ZZ(0)  \to M^\vee \otimes M$ the arrows of $\mathcal{MM}_{\leq 1}(K)$ characterizing this dual. The Cartier dual of $M$ is $M^*= M^\vee \otimes \ZZ(1)$.
If $M_1,M_2$ are two 1-motives, we set
\begin{equation}\label{eq:BiextHom}
\Hom_{\mathcal{MM}_{\leq 1}(K)}(M_1 \otimes M_2, M_3):= \mathrm{Biext}^1 (M_1,M_2; M_3)
\end{equation}
where $\mathrm{Biext}^1 ((M_1,M_2;M_3)$ is the abelian group of isomorphism classes of biextensions of $(M_1,M_2)$ by $M_3$.
In particular the isomorphism class of the Poincar\'e biextension $\mathcal{P}$ of 
$(A,A^*)$ by $\GG_m$ is the Weil pairing 
$P_\mathcal{P} : A \otimes A^* \to \ZZ(1)$ of $A.$

The tannakian sub-category $<M>^\otimes$ generated by the 1-motive $M$ is the full sub-category of $\mathcal{MM}_{\leq 1}(K)$ whose objects are sub-quotients of direct sums of $M^{\otimes \; n}  \otimes M^{\vee \; \otimes \; m}$, and whose fibre functor is the restriction of the fibre functor of $\mathcal{MM}_{\leq 1}(K)$ to $<M>^\otimes$.
 Because of the tensor product of $<M>^\otimes$, we have the notion of commutative Hopf algebra in the category $\Ind <M>^\otimes$ of Ind-objects of $<M>^\otimes$, and this allows us to define the category of affine $<M>^\otimes$-group schemes, just called \textit{motivic affine group schemes}, as the opposite of the category of  
commutative Hopf algebras in $\Ind <M>^\otimes.$
The Lie algebra of a motivic affine group scheme is a pro-object $\rm L$ of $\langle M \rangle^\otimes$ endowed with a Lie algebra structure,
 i.e. $\rm L$ is endowed with an anti-symmetric application
$[\, , \,]: {\rm L} \otimes {\rm L} \to {\rm L}$
satisfying the Jacobi identity.

 The \textit{motivic Galois group} 
 $\Galmot (M)$ of $M$ is the fundamental group of the tannakian category 
 $< M >^\otimes$ generated by  $M$, i.e.
 the motivic affine group scheme ${\rm Sp}( \Lambda),$ where $ \Lambda$ is the commutative Hopf algebra  of $<M>^\otimes$ which is universal for the following property: 
for any object $X$ of $<M>^\otimes,$ it exists a morphism  
\begin{equation}\label{eq:lambdaX}
\lambda_X:  X^{\vee} \otimes X \longrightarrow   \Lambda 
\end{equation}
functorial on  $X$, i.e. such that for any morphism $f: X \to Y$ in $<M>^\otimes$ the diagram 
 \[
\begin{matrix} 
Y^{\vee} \otimes X&{\buildrel f^t \otimes 1 \over \longrightarrow}&	X^{\vee} \otimes X \cr
{\scriptstyle 1 \otimes f}\downarrow \quad \quad & & \quad \quad \downarrow
		{\scriptstyle \lambda_X}\cr
Y^{\vee} \otimes Y & {\buildrel \lambda_Y \over \longrightarrow} & \Lambda
\end{matrix}	
\]
is commutative. The universal property of $\Lambda$ is that for any object $U$ of  $<M>^\otimes$, the map
\begin{align}
\nonumber	{\Hom}(\Lambda, U) & \longrightarrow  \big\{ u_X:
		X^{\vee} \otimes X \to U, ~~ {\rm {functorial~ on~}} X \big\}  \\
\nonumber	f & \longmapsto  f  \circ \lambda_X
\end{align}
is bijective. The morphisms (\ref{eq:lambdaX}), which can be rewritten as $ X \to X \otimes \Lambda$,
define the action of the motivic Galois group
$\Galmot (M)$ on each object $X$ of $<M>^\otimes$.

If $\omega_\QQ$ is
the fibre functor Hodge realization of the tannakian category $<M>^\otimes$, 
$\omega_\QQ (\Lambda)$ is the Hopf algebra whose spectrum ${\rm Spec} (\omega (\Lambda))$ is the $\QQ$-group scheme 
$ {\underline {\rm Aut}}^{\otimes}_\QQ(\omega_\QQ)$, i.e. the Mumford-Tate group $\mathrm{MT}(M)$ of $M$. In other words, the motivic Galois group of $M$ is \textit{the geometric interpretation} of the Mumford-Tate group of $M$.
By \cite[Thm 1.2.1]{A19} these two group schemes coincides, and in particular they have the same dimension
\begin{equation}\label{dimGalMT}
\dim  \Galmot (M) = \dim \mathrm{MT}(M).
\end{equation}

Let  $M=[u:X \to G]$ be a 1-motive defined over $K$, with $G$  an extension of an abelian variety $A$ by a torus $T$.
The weight filtration $\W_{\bullet}$ of $M$ induces a filtration on its motivic Galois group $\Galmot (M)$ (\cite[Chp IV \S 2]{S72}):
\vskip 0.3 true cm
\par\noindent  $  \W_{0}(\Galmot(M))=\Galmot(M) $
\vskip 0.3 true cm
 \par\noindent $  \W_{-1}(\Galmot(M))= \big\{ g \in \Galmot(M) \, \, \vert \, \,
(g - id)M \subseteq   \W_{-1}(M)  ,(g - id) \W_{-1}(M)
\subseteq \W_{-2}(M),$
\par\noindent $ (g - id) \W_{-2}(M)=0 \big\} , $
\vskip 0.3 true cm
\par\noindent  $  \W_{-2}(\Galmot(M))=\big\{ g \in \Galmot(M) \, \, \vert \, \,
(g - id) M \subseteq \W_{-2}(M),  (g - id)  \W_{-1}(M) =0 \big\}, $
\vskip 0.3 true cm
\par\noindent $   \W_{-3}(\Galmot(M))=0.$
\vskip 0.3 true cm
\par\noindent Clearly $  \W_{-1}(\Galmot(M))$  is unipotent. Denote by $\UR (M)$ the unipotent radical of $\Galmot(M)$.

Consider the graduated 1-motive
  \[\widetilde{M}= \Gr_*^\W (M) = X+A+T \]
  associated to $M$ and let $<\widetilde{M}>^\otimes$ be the tannakian sub-category of $<M>^\otimes$ generated by $\widetilde{M}$.
  The functor ``take the graduated'' $\Gr_*^\W : <M>^\otimes \twoheadrightarrow <\widetilde{M}>^\otimes$, which is a projection,  induces the inclusion of motivic affine group schemes
  \begin{equation} \label{eq:Gr_0}
 \Galmot(\widetilde{M}) \hookrightarrow \Gr_*^\W \Galmot(M).
  \end{equation}

  \begin{lemma}\label{eq:dimGr0}
  	Let  $M=[u:X \to G]$ be a 1-motive defined over $K$, with $G$  an extension of an abelian variety $A$ by a torus $T$. The quotient ${\Gr}_{0}^{\W}(\Galmot(M))$
  	is reductive and the inclusion of motivic group schemes (\ref{eq:Gr_0})
  	identifies $\Galmot(\widetilde{M}) $ with this quotient.
  	
  	Moreover, if $X= \ZZ^r$ and $T= \GG_m^s$
  	\[
\dim {\Gr}_{0}^{\W}\big(\Galmot(M)\big)= 	\dim\Galmot (\widetilde{M}) = 
  	\begin{cases}
  	\dim \Galmot (A) & \mbox{if } A \not= 0,\\
  	1 & \mbox{if } A=0, T \not=0 ,\\
  	0 & \mbox{if } A=T=0.
  	\end{cases} 
  	\]
  \end{lemma}
  
  \begin{proof}
  By a motivic analogue of \cite[\S 2.2]{By83}, 
  ${\Gr}_{0}^{\W}(\Galmot(M))$ acts via $\gal$ on ${\Gr}_{0}^{W}(M)$,
  by homotheties on ${\Gr}_{-2}^{\W}(M)$, and
  its image in the group of authomorphisms of
  ${\Gr}_{-1}^{\W}(M)$
  is the motivic Galois group $\Galmot(A)$ of the abelian variety 
  $A $ underlying $M$.
  Therefore ${\Gr}_{0}^{\W}(\Galmot(M))$
  is reductive, and via the inclusion (\ref{eq:Gr_0})
  it coincides with $\Galmot(\widetilde{M}) .$
  To conclude, observe that 
 	$\Lie \, \Galmot(\GG_m)= \GG_m$ which has dimension 1, and $\Galmot(\ZZ)= \mathrm{Sp}(\ZZ(0))$ which has dimension 0.
  \end{proof}

The inclusion $<\widetilde{M}>^\otimes \hookrightarrow <M>^\otimes $ of tannakian categories induces the following surjection of motivic affine group schemes
\begin{equation} \label{eq:RestrictionGr_0}
\Galmot (M) \twoheadrightarrow  \Galmot (\widetilde{M})
\end{equation}
 which is the restriction $g \mapsto 	g_{|\widetilde{M} }.$ As an immediate consequence of the above Lemma we have

\begin{corollary}\label{eq:DecomDim}
	Let  $M=[u:X \to G]$ be a 1-motive defined over $K$. Then
\[\W_{-1}(\Galmot(M))=\ker \big[\Galmot (M) \twoheadrightarrow  \Galmot (\widetilde{M}) \big]. \]
In particular, $\W_{-1}(\Galmot(M))$ is the unipotent radical $ \UR (M)$ of 
$\Galmot(M)$ and 
	\[
	\dim \Galmot (M) = \dim \Galmot (\widetilde{M}) + \dim \UR (M).
	\]
\end{corollary}

Observe that we can prove the equality $ \W_{-1}(\Galmot(M))=\ker \big[\Galmot (M) \twoheadrightarrow  \Galmot (\widetilde{M}) \big]$ directly using the definition of the weight filtration:
	 \[
	\begin{aligned}
	g \in  \W_{-1}(\Galmot(M))  & \Longleftrightarrow (g - id) {\Gr}_{0}^{\W}(M) =0, (g - id) {\Gr}_{-1}^{\W}(M) =0,(g - id) {\Gr}_{-2}^{\W}(M) =0
	\\
	& \Longleftrightarrow 	g_{| {\Gr}_{*}^{\W}(M) } = \id, \;  \; \mathrm{i.e.} \;  \;  g= \id \;\;   \mathrm{in}\; \; \Galmot(\widetilde{M}).
	\end{aligned}\]

The inclusion $<M + M^\vee /\W_{-2} (M + M^\vee) >^\otimes \hookrightarrow <M>^\otimes $ of tannakian categories induces the following surjection of motivic affine group schemes
\begin{equation} \label{eq:Gr_1}
\Galmot (M) \twoheadrightarrow  \Galmot \big(M + M^\vee /\W_{-2} (M + M^\vee)\big)
\end{equation}
 which is the restriction $g \mapsto 	g_{|M + M^\vee /\W_{-2} (M + M^\vee) }.$

\begin{lemma}\label{eq:DecomRU}
	Let  $M=[u:X \to G]$ be a 1-motive defined over $K$. Then
	\[\W_{-2}(\Galmot(M))=\ker \big[\Galmot (M) \twoheadrightarrow   \Galmot (M + M^\vee /\W_{-2} (M + M^\vee)) \big]. \]
	In particular, the quotient ${\Gr}_{-1}^{\W}(\Galmot(M))$ of the unipotent radical $ \UR (M)$  is  the unipotent radical $\W_{-1} \big( \Galmot (M + M^\vee /\W_{-2} (M + M^\vee))\big)$  of 
	$\Galmot\big( M + M^\vee /\W_{-2} (M + M^\vee)\big)$. 
\end{lemma}

\begin{proof} Using the definition of the weight filtration, we have:
\[
\begin{aligned}
g \in  \W_{-2}(\Galmot(M))  & \Longleftrightarrow (g - id) M/\W_{-2}(M)  =0,\; (g - id)\W_{-1}(M)  =0
\\
& \Longleftrightarrow 	g_{|  M/\W_{-2}(M) } = \id, \;	g_{|  M^\vee/\W_{-2}(M^\vee) } = \id \\
& \Longleftrightarrow  g= \id \; \;  \mathrm{in}\;\; \Galmot (M + M^\vee /\W_{-2} (M + M^\vee)).
\end{aligned}\]
Since the  surjection of motivic affine group schemes (\ref{eq:Gr_1}) respects the weight filtration,  $\W_{-2}(\Galmot(M))$ is in fact the kernel of $\W_{-1} (\Galmot (M)) \twoheadrightarrow  \W_{-1} (\Galmot (M + M^\vee /\W_{-2} (M + M^\vee))) .$ Hence we get the second statement.
\end{proof}

\par\noindent From the definition of weight filtration, we observe that
\[ \W_{-2}(\Galmot(M)) \subseteq \uHom (X,Y(1)) \cong X^\vee \otimes Y (1). \]
By the above Lemma, we have that
\[{\Gr}_{-1}^{\W}(\Galmot(M)) \subseteq \uHom (X+ Y^\vee,A+A^*) \cong X^\vee  \otimes A+Y \otimes A^*. \]

 In order to compute the dimension of the unipotent radical $\UR (M) $ of $\Galmot (M )$ we use notations of \cite[\S 3]{B03} that we recall briefly.
Let  $(X,Y^\vee, A,A^*, v:X \to A, 
v^*:Y^\vee \to A^*, \psi:X \otimes Y^\vee \to
(v \times v^*)^* \mathcal{P})$ be the 7-tuple defining the 1-motive 
$M=[u:X \to G]$ over $K$, where $G$ an extension of $A$ by the torus $Y(1)$.
Let 
\[E=\W_{-1}( {\underline {\End}}(\widetilde{M})).\]
It is the direct sum of the pure motives $E_{-1}= X^\vee \otimes A + A^\vee \otimes Y(1)$ and $E_{-2}= X^\vee \otimes Y(1)$ of weight -1 and -2. As observed in \cite[\S 3]{B03}, the composition of endomorphisms 
furnishes a ring structure to $E$ given by the arrow  $P: E \otimes E \to E$ of $\langle M \rangle^\otimes$ whose only non trivial component 
is 
$$E_{-1} \otimes E_{-1} \longrightarrow (X^\vee \otimes A) \otimes (A^* \otimes Y) 
\longrightarrow {\ZZ}(1) \otimes X^\vee \otimes Y = E_{-2},$$
where the first arrow is the projection from  $ E_{-1} \otimes E_{-1}$ to $(X^\vee \otimes A) \otimes (A^* \otimes Y) $ and the second arrow is the Weil pairing $P_{\mathcal{P}}:
A \otimes A^* \to \ZZ(1)$ of $A.$

Because of the definition (\ref{eq:BiextHom}) the product
$P:E_{-1} \otimes E_{-1} \to E_{-2}$
defines a biextension $\mathcal{B}$ of $(E_{-1},E_{-1})$ by 
$E_{-2} $, whose pull-back $ d^* \mathcal{B}$ via the diagonal morphism $d:  E_{-1} \to E_{-1} \times E_{-1}$ is a $\Sigma - X^\vee \otimes Y  (1)$-torsor over $E_{-1}$. By \cite[Lem 3.3]{B03} this $\Sigma - X^\vee \otimes Y  (1)$-torsor $ d^* \mathcal{B}$ induces a Lie bracket 
$[\, ,\,]: E \otimes E \to E$ on $E$ which becomes therefore a Lie algebra.

The action of $E=W_{-1}( {\underline {\End}}(\widetilde{M}))$ on $\widetilde{M}$ is given by the arrow 
$ E \otimes \widetilde{M} \to \widetilde{M}$ of $\langle M \rangle^\otimes$
 whose only non trivial components are
\begin{align}
\label{eq:alpha1}	\alpha_1:& (X^\vee \otimes A) \otimes X  \longrightarrow A \\
\nonumber	\alpha_2:& (A^* \otimes Y) \otimes A  \longrightarrow Y(1) \\
\nonumber 	\gamma :& (X^\vee \otimes Y(1)) \otimes X  \longrightarrow Y(1),
\end{align}
where the first and the last arrows are induced by 
$ev_{X^\vee}: X^\vee \otimes X \to \ZZ(0)$, 
while the second one is $\rk (Y)$-copies of the Weil pairing $P_{\mathcal{P}}:
A \otimes A^* \to \ZZ(1)$ of $A$. 
By \cite[Lem 3.3]{B03},
via the arrow $(\alpha_1, \alpha_2, \gamma ):  E \otimes \widetilde{M} \to \widetilde{M}$, the 1-motive $\widetilde{M}$ is in fact a  $(E,[,])$-Lie module. 

As observed in \cite[Rem 3.4 (3)]{B03} $E$ acts also on the Cartier dual $\widetilde{M}^*= Y^\vee + A^* + X^\vee(1)$ of $\widetilde{M}$ and this action is given by the arrows
\begin{align}
\label{eq:alpha2*}		\alpha_2^*:& (A^* \otimes Y) \otimes Y^\vee \longrightarrow A^* \\
\nonumber\alpha_1^*:& (X^\vee \otimes A) \otimes A^*  \longrightarrow X^\vee(1) \\
\nonumber \gamma^* :& (X^\vee \otimes Y(1)) \otimes Y^\vee \longrightarrow X^\vee(1),
\end{align}
where $\alpha_2^*$ et $\gamma^*$ are projections, while $\alpha_1^*$ 
is $\rk (X^\vee)$-copies of the Weil pairing $P_{\mathcal{P}}:
A \otimes A^* \to \ZZ(1)$ of $A$.

Via the arrows
$\delta_{ X^\vee}: \ZZ(0) \to X \otimes X^\vee$ et
$\delta_{ Y}: \ZZ(0) \to Y^\vee \otimes Y$, to have the morphisms 
$v: X \to A$ and $v^*: Y^\vee \to A^*$ underlying the 1-motive $M$ is the same thing as to have the morphisms 
$V: \ZZ(0) \to A \otimes X^\vee $ and 
$V^*: \ZZ(0) \to A^* \otimes Y,$ i.e. to have a  point 
\[b=(b_1,b_2) \in E_{-1}(K)= A \otimes X^\vee(K)+A^* \otimes Y(K). \]

Fix now an element $(x,y^\vee)$ in the character group $X \otimes  Y^\vee$ of the torus $X^\vee \otimes Y(1)$.
By construction of the point $b$, it exists an element $(s,t) \in X \otimes  Y^\vee (K)$ such that
\begin{align}
\nonumber v(x) &= \alpha_1(b_1,s) \in A(K)	 \\
\nonumber v^*(y^\vee)& =\alpha_2^*(b_2,t)\in A^*(K).
\end{align}
Let $i^*_{x,y^\vee} d^*\mathcal{B}$ be the pull-back of $d^* \mathcal{B}$
via the inclusion $i_{x,y^\vee}: \{ (v(x),v^*(y^\vee) )\}  
\hookrightarrow E_{-1}$ in $E_{-1}$ of the abelian sub-variety generated by the point 
$ (v(x),v^*(y^\vee) )$.
The push-down 
$(x,y^\vee)_*i^*_{x,y^\vee} d^*\mathcal{B}$ 
of $i^*_{x,y^\vee} d^*\mathcal{B}$ via the character
$(x,y^\vee):X^\vee \otimes Y(1)\to \ZZ(1)$ is a $\Sigma-\ZZ(1)$-torsor 
over $ \{ (v(x),v^*(y^\vee))\}:  $
\[\begin{matrix} 
	(x,y^\vee)_*i^*_{x,y^\vee} d^*\mathcal{B}&  \longleftarrow & i^*_{x,y^\vee} d^*\mathcal{B} &  \longrightarrow & d^* \mathcal{B}  \\
	\downarrow & &\downarrow & & \downarrow  \\
	\{ (v(x),v^*(y^\vee)) \}    & =  &
	\{ (v(x),v^*(y^\vee)) \}    &  
	{\buildrel i_{x,y^\vee} \over \longrightarrow}   & E_{-1} 
\end{matrix} \]
To have the point $\psi(x,y^\vee)$ is equivalent to have a point
 $(\widetilde{b})_{x,y^\vee}$ of
 $(x,y^\vee)_*i^*_{x,y^\vee} d^*\mathcal{B}$ over $ (v(x),v^*(y^\vee))$, and so to have the trivialization
$\psi$ is equivalent to have a point 
\[\widetilde{b} \in (d^*\mathcal{B})_{b}\]
 in the fibre of $d^*\mathcal{B}$ over $b=(b_1,b_2).$ 

Consider now the following pure motives:
\begin{enumerate}
	\item Let $B$ be the \textit{smallest} abelian sub-variety (modulo isogenies)
	of  $X^\vee  \otimes A+A^* \otimes Y$ which contains the point 
	$b=(b_1,b_2) \in X^\vee  \otimes A (K) +
	A^* \otimes Y (K) $. The pull-back
	 $i^*d^* \mathcal{B}$ of $d^* \mathcal{B}$ via the inclusion
	 $i: B \hookrightarrow E_{-1}$ of $B$ 
	on $E_{-1}$, is a $\Sigma-X^\vee \otimes Y(1)$-torsor
	over $B$.
	\item Let $Z_1$ be the \textit{smallest} $\gal$-sub-module of
	$X^\vee \otimes Y$  such that the torus $Z_1(1)$ contains the image of the Lie bracket $[\, ,\,]: B \otimes B \to X^\vee \otimes Y(1)$.
	The push-down $p_*i^*d^* \mathcal{B}$ of the $\Sigma-X^\vee \otimes Y(1)$-torsor
	$i^*d^* \mathcal{B}$ via the projection
	$p:X^\vee \otimes Y(1) \twoheadrightarrow
	(X^\vee \otimes Y/ Z_1)(1)$ is the trivial $\Sigma-(X^\vee \otimes Y/ Z_1)(1)$-torsor over $B$, 
	i.e. 
	\[p_*i^*d^* \mathcal{B}= B \times (X^\vee \otimes Y/ Z_1)(1).\] 
	Note by $\pi: p_*i^*d^* \mathcal{B} \twoheadrightarrow (X^\vee \otimes Y/ Z_1)(1)$
	the canonical projection. We still note $\widetilde{b}$ the points of $i^*d^* \mathcal{B}$
	and of $p_*i^*d^* \mathcal{B}$ living over $b \in B$.
	\item Let $Z$ be the \textit{smallest} $\gal$-sub-module of
	$X^\vee \otimes Y$ containing
 $Z_1$ and such that
	the sub-torus $(Z/ Z_1)(1)$ of $(X^\vee \otimes Y/ Z_1)(1)$ contains
	$\pi (\widetilde{b}) $. 
\end{enumerate}

Let $A_\CC$ be the abelian variety defined over $\CC$ obtained from $A$ extending the scalars from $K$ to the complexes. Denote by $g$ the dimension of $A$.
Consider the abelian exponential
\[	\exp_{A}: \Lie A_\CC \longrightarrow  A_\CC \]
whose kernel is the lattice $\HH_1(A_\CC(\CC),\ZZ),$ and denote by $\log_A$ an abelian logarithm of $A$, that is a choice of an inverse map of $\exp_{A}$. Consider the composite
\[ P_\mathcal{P} \circ (v \times v^*): X \otimes Y^\vee \longrightarrow \ZZ(1)\]
where $P_\mathcal{P}: A \otimes A^* \to \ZZ(1)$ is the Weil pairing of $A.$
Since we work modulo isogenies, we identify the abelian variety $A$ with its Cartier dual $A^*$.
Let $\omega_1, \dots , \omega_g $ be differentials of the first kind which build a basis of 
the $K$-vector space $\HH^0(A, \Omega^1_A)$ of holomorphic differentials, and let $\eta_1, \dots , \eta_g $ be differentials of the second kind which build a basis of 
the $K$-vector space $\HH^1(A, \mathcal{O}_A)$ of differentials of the second kind modulo holomorphic differentials and exact differentials. As in the case of elliptic curves, the first De Rham cohomology group $\HH^1_\dR(A)$ of the abelian variety $A$ is the direct sum $\HH^0(A, \Omega^1_A) \oplus \HH^1(A, \mathcal{O}_A)$ of these two vector spaces and it has dimension $2g$.
Let $\gamma_1, \dots, \gamma_{2g}$  be closed paths which build a basis of the $\QQ$-vector space $\HH_1(A_\CC,\QQ)$. For $n=1, \dots, g$ and $m=1,\dots, 2g$, the abelian integrals of the first kind
 $ \int_{\gamma_m} \omega_n = \omega_{nm}$  are the \textit{periods} of the abelian variety $A$, and  the abelian integrals of the second kind $ \int_{\gamma_m} \eta_n = \eta_{nm}$ are the \textit{quasi-periods} of $A$.

\begin{theorem}\label{eq:dimUR}
	Let $M=[u:X \to G]$ be a 1-motive defined over $K$, with $G$  an extension of an abelain variety $A$ by a torus $Y(1)$. Denote by 
	$ F=\End( A) \otimes_\ZZ \QQ$ the field of endomorphisms of the abelian variety $A.$ 
	Let $x_1, \dots, x_{\mathrm{rk}(X)}$ be generators of the character group $X$ and let $y^\vee_1, \dots , y^\vee_{\mathrm{rk}(Y^\vee)}$ be generators of the character group $Y^\vee.$
	Then
	\[
 \dim_{\QQ} \UR( M)=\]
\[   2 \dim_{F} \mathcal{A}b \mathcal{L}og \; \im (v,v^*) +
  \dim_{\QQ}  \mathcal{L}og  \;\im (P_\mathcal{P} \circ (v \times v^*)) + 
\dim_{\QQ}  \mathcal{L}og \; \im (\psi_{| \ker (P_\mathcal{P} \circ (v \times v^*))})
	\]
	where 
	\begin{itemize}
		\item 	$ \mathcal{A}b\mathcal{L}og \; \im(v,v^*)$ is the $F$-sub-vector space of $\CC /  (\sum_{n=1, \dots, g \atop m=1, \dots, 2g}F \, \omega_{nm})
		$  generated by	the abelian logarithms 
		$\{ \log_A v(x_k), \log_A v^*(y^\vee_i) \}_{ k=1,  \ldots, \mathrm{rk}(X) \atop  i=1,  \ldots, \mathrm{rk}(Y^\vee)}$ ;
		\item  $ \mathcal{L}og  \; \im ( P_\mathcal{P} \circ (v \times v^*))$ is the $\QQ\,$-sub-vector space of $\CC / 2 i \pi \QQ$	generated by the logarithms 
		\par\noindent  $\{ \log P_\mathcal{P}(v(x_k),v^*(y^\vee_i) ) \}_{ k=1,  \ldots, \mathrm{rk}(X) \atop  i=1,  \ldots, \mathrm{rk}(Y^\vee)}$;  
		\item  $  \mathcal{L}og  \; \im (\psi_{| \ker (P_\mathcal{P} \circ (v \times v^*))})$ is the $\QQ\,$-sub-vector space of $\CC / 2 i \pi \QQ$ generated by the logarithms
	 $\{ \log \psi(x_{k'},y^\vee_{i'} ) \}_{(x_{k'},y^\vee_{i'} ) \in \ker (P_\mathcal{P} \circ (v \times v^*))  \atop 
	 1\leq {k'} \leq \mathrm{rk}(X), \;  1\leq {i'} \leq 	\mathrm{rk}(Y^\vee)	} .$
	\end{itemize}

\end{theorem}

\begin{proof}  By the main theorem of \cite[Thm 0.1]{B03}, the unipotent
radical $W_{-1} (\Lie {\Galmot}(M))$ is the semi-abelian variety extension of $B$ by $Z(1)$ defined by the adjoint action of the Lie algebra $(B,Z(1), [\, , \,])$ over $B+Z(1).$ Since the tannakian category $<M>^\otimes$ has  rational coefficients, we have that
$ \dim_{\QQ} W_{-1} ( {\Galmot}(M)) = 2 \dim B + \dim Z(1) $.
 Concerning the abelian part
 \[\dim B= 	\dim_F \mathcal{A}b\mathcal{L}og \; \im(v,v^*).\]
 On the other hand, for the toric part $\dim Z(1) = \dim (Z/ Z_1)(1) + Z_1(1)$ by construction. Because of the explicit description of the Lie bracket
$[\, ,\,]: B \otimes B \to X^\vee \otimes Y(1)$ given in \cite[(2.8.4)]{B03},
\[ \dim Z_1(1) =  \dim_{\QQ}  \mathcal{L}og  \;\im (P_\mathcal{P} \circ (v \times v^*)). \]
Finally by construction 
\[ \dim(Z/ Z_1)(1) =  \dim_{\QQ} \mathcal{L}og \; \im (\psi_{| \ker (P_\mathcal{P} \circ (v \times v^*))}) .\]
\end{proof}

\begin{remark} The dimension of the quotient ${\Gr}_{-1}^{\W}(\Galmot(M))$ of the unipotent radical $
\UR( M)$ is twice the dimension of the abelian sub-variety $B$ of $X^\vee  \otimes A+A^* \otimes Y,$ that is 
\[ \dim_\QQ {\Gr}_{-1}^{\W}(\Galmot(M)) = 2 \dim_F \mathcal{A}b\mathcal{L}og \; \im(v,v^*).\]
The dimension of  $ W_{-2} ( {\Galmot}(M))$ is the dimension of the sub-torus $Z(1)$ of $X^\vee \otimes Y(1)$, that is 
\[  \dim_{\QQ} W_{-2} ( {\Galmot}(M)) =  \dim_{\QQ}  \mathcal{L}og  \;\im (P_\mathcal{P} \circ (v \times v^*)) + \dim_{\QQ} \mathcal{L}og \; \im (\psi_{| \ker (P_\mathcal{P} \circ (v \times v^*))})  \]
\end{remark}

\begin{remark} A 1-motive $M=[u:X \to G]$ defined over $K$ is said to be \textit{deficient} if $ \W_{-2}(\Galmot(M))=0.$ In \cite{JR} Jacquinot and Ribet construct such a 1-motive in the case $\mathrm{rk}(X)=\mathrm{rk}(Y^\vee)=1$. By the above Theorem we have that $M$ is deficient if and only if for any $(x,y^{\vee}) \in X \otimes Y^{\vee}$, 
\[ P_\mathcal{P} (v(x), v^*(y^{\vee})) =1 \quad \mathrm{and} \quad \psi_{| \ker (P_\mathcal{P} \circ (v \times v^*))}(x,y^{\vee}) =1,\]
that is if and only if the two arrows $P_\mathcal{P} \circ (v \times v^*): X \otimes Y^{\vee} \to  \ZZ(1)$ and 
$\psi_{| \ker (P_\mathcal{P} \circ (v \times v^*))} : X \otimes Y^{\vee} \to  \ZZ(1)$ are the trivial arrow.
\end{remark}

Now let $M=[u:{\ZZ}^r \to G]$ be a 1-motive defined over $K$, with $G$ an extension of a product $\Pi^n_{j=1} \cE_j $ of pairwise not isogenous elliptic curves by the torus $\GG_m^s.$ We go back to the notation used in Section \ref{periods}.
Denote by $\pr_h:\Pi^n_{j=1} \cE_j \to  \cE_h$ and $\pr_h^*:\Pi^n_{j=1} \cE_j^* \to  \cE_h^*$ the projections into the $h$-th elliptic curve and consider the composites 
$ v_h= \pr_h \circ v:\ZZ^r \rightarrow \cE_h$ and $v^*_h=\pr_h^* \circ v^* :\ZZ^s \rightarrow  \cE_h^*.$ 
Let $\mathcal{P}$ be the Poincar\'e biextension of $(\Pi^n_{j=1} \cE_j, \Pi^n_{j=1} \cE_j^*)$ by $\GG_m$
and let $\mathcal{P}_j$ be the Poincar\'e biextension of $( \cE_j, \cE_j^*)$ by $\GG_m$.
The category of biextensions is additive in each variable, and so we have that 
$P_\mathcal{P} = \Pi^n_{j=1}P_{\mathcal{P}_j}$, where $P_{\mathcal{P}_j}: \cE_j \otimes \cE_j^*\to  \ZZ(1)$ is the Weil pairing of the elliptic curve $\cE_j$.

\begin{corollary}\label{eq:dimGalMot}
	Let $M=[u:{\ZZ}^r \to G]$ be a 1-motive defined over $K$, with $G$ an extension of a product $\Pi^n_{j=1} \cE_j $ of pairwise not isogenous elliptic curves by the torus $\GG_m^s.$  Denote by 
	$ k_j=\End( \cE_j) \otimes_\ZZ \QQ$ the field of endomorphisms of the elliptic curve $\cE_j$ for $j=1, \dots, n.$ Let $x_1, \dots, x_r$ be generators of the character group ${\ZZ}^r$ and let $y^\vee_1, \dots , y^\vee_{s}$ be generators of the character group ${\ZZ}^{s}.$ Then
	\[
	\dim_\QQ \Galmot(M) =  4 \sum_{j=1}^n (\dim_{\QQ} k_j)^{-1}-n+1 + \sum_{j=1}^n  2 \dim_{k_j} \mathcal{A}b \mathcal{L}og \; \im (v_j,v^*_j) + \]
	\[
  \dim_{\QQ}  \mathcal{L}og  \;\im (P_\mathcal{P} \circ (v \times v^*)) + 
\dim_{\QQ}  \mathcal{L}og \; \im (\psi_{| \ker (P_\mathcal{P} \circ (v \times v^*))})
\]
\begin{itemize}
	\item $\mathcal{A}b \mathcal{L}og \; \im (v_j,v^*_j)$ is the $k_j$-sub-vector space of $\CC /  k_j \, \omega_{j1}+  k_j \, \omega_{j2}$ generated by the elliptic logarithms 
	$\{ p_{jk}, q_{ji} \}_{ k=1,  \ldots, r \atop  i=1,  \ldots, s}$ of the points $\{ P_{jk}, Q_{ji} \}_{ k=1,  \ldots, r \atop  i=1,  \ldots, s}$ for $j=1,  \ldots, n;$
	\item  $ \mathcal{L}og  \; \im ( P_\mathcal{P} \circ (v \times v^*))$ is the $\QQ\,$-sub-vector space of $\CC / 2 i \pi \QQ$	generated by the logarithms 
		\par\noindent  $\{ \log P_{\mathcal{P}_j}(P_{jk}, Q_{ji} ) \}_{ k=1,  \ldots,r, \; \; i=1,  \ldots, s \atop j=1, \ldots, n}$;  
		\item  $  \mathcal{L}og  \; \im (\psi_{| \ker (P_\mathcal{P} \circ (v \times v^*))})$ is the $\QQ\,$-sub-vector space of $\CC / 2 i \pi \QQ$ generated by the logarithms
	 $\{ \log \psi(x_{k'} ,y^\vee_{i'} ) \}_{(x_{k'} ,y^\vee_{i'} ) \in \ker (P_{\mathcal{P}_j} \circ (v_j\times v^*_j))  \atop 
		 1\leq {k'} \leq r, \;  1\leq {i'} \leq 	s, \; j=1,  \ldots, n	} .$
\end{itemize}
	\end{corollary}

\begin{proof} Since the elliptic curves are pairwise not isogenous, by 
\cite[\S 2]{Moonen} and (\ref{dimGalMT}) we have that
\[\dim  \Galmot\big(\Pi_{j=1}^{n} \cE_j \big)=4 \; \sum_{j=1}^n 
(\dim_\QQ k_j)^{-1}-n+1.\]
Therefore putting together Corollary \ref{eq:DecomDim}, Lemma \ref{eq:dimGr0} and Theorem \ref{eq:dimUR} we can conclude.
\end{proof}

\begin{remark} We can express the dimension of the motivic Galois group of a product of elliptic curves also as
	 $ 3 n_1 +n_2 +1,$ where $n_1$ is the number of elliptic curves without complex multiplication and $n_2$ is the number of elliptic curves with complex multiplication. Therefore
	 	\[
	 \dim \Galmot(M) =\dim\UR (M) + 3 n_1 +n_2 +1
	 \]
\end{remark}

\section{The 1-motivic elliptic conjecture}\label{conjecture}

\pn \textbf{The 1-motivic elliptic conjecture}
\par\noindent Consider
\begin{itemize}
	\item  $n$ elliptic curves $\cE_1, \dots, \cE_n$ pairwise not isogenous. 
	For $j=1,  \ldots, n,$ denote by $ k_j=\End( \cE_j) \otimes_\ZZ \QQ$ the field of endomorphisms of $\cE_j$ and let
	 $g_{2j}=60 \; \mathrm{G}_{4j}$ and $g_{3j}=140 \; \mathrm{G}_{6j}$, where $\mathrm{G}_{4j}$ and $\mathrm{G}_{6j}$ are the Eisenstein series relative to the lattice $ \HH_1(\cE_j(\CC),\ZZ)$ of weight 4 and 6 respectively;
	\item  $s$ points $Q_i= (Q_{1i},\dots,Q_{ni})$ of  $\Pi^n_{j=1} \cE_j^*(\CC)$ for $i=1,\ldots, s$.
	These points determine an extension $G$ of $\Pi^n_{j=1} \cE_j $ by $
	{\GG}_m^s$;
	\item  $r$ points $R_1,\dots,R_r $ of $G(\CC)$.  Denote by
	$(P_{1k},\dots,P_{nk})  \in \Pi^n_{j=1} \cE_j(\CC)$ the projection of the point $R_k$ on $\Pi^n_{j=1} \cE_j$ for $k=  \ldots, r.$
\end{itemize}
Then

\[
\mathrm{tran.deg}_{\QQ}\, \QQ  \Big(2i \pi,
g_{2j},g_{3j}, Q_{ji},R_{k},\omega_{j1},\omega_{j2},\eta_{j1},
\eta_{j2},
p_{jk},\zeta_j(p_{jk}),\]
\[
\eta_{j1} q_{ji} - \omega_{j1} \zeta_j(q_{ji}),\eta_{j2} q_{ji} - \omega_{j2} \zeta_j(q_{ji}) , 
\log f_{q_{ji}}(p_{jk}) +l_{jik}
{\Big)}_{j=1,\ldots,n \;\;  i=1,\ldots,s   \atop
	k=1,\dots,r  } \geq \]
	\[ 4 \sum_{j=1}^n (\dim_{\QQ} k_j)^{-1}-n+1 +  \sum_{j=1}^n  2 \dim_{k_j} \mathcal{A}b \mathcal{L}og \; \im (v_j,v^*_j) + \]
	\[  \dim_{\QQ}  \mathcal{L}og  \;\im (P_\mathcal{P} \circ (v \times v^*)) + 
\dim_{\QQ}  \mathcal{L}og \; \im (\psi_{| \ker (P_\mathcal{P} \circ (v \times v^*))})
\]
where
\begin{itemize}
	\item $\mathcal{A}b \mathcal{L}og \; \im (v_j,v^*_j)$ is the $k_j$-sub-vector space of $\CC /  k_j \, \omega_{j1}+  k_j \, \omega_{j2}$ generated by the elliptic logarithms 
	$\{ p_{jk}, q_{ji} \}_{ k=1,  \ldots, r \atop  i=1,  \ldots,s}$ of the points $\{ P_{jk}, Q_{ji} \}_{ k=1,  \ldots, r \atop  i=1,  \ldots, s}$ for $j=1,  \ldots, n;$
	\item  $ \mathcal{L}og  \; \im ( P_\mathcal{P} \circ (v \times v^*))$ is the $\QQ\,$-sub-vector space of $\CC / 2 i \pi \QQ$	generated by the logarithms 
		\par\noindent  $\{ \log P_{\mathcal{P}_j}(P_{jk}, Q_{ji} ) \}_{ k=1,  \ldots, r, \; \; i=1,  \ldots, s \atop j=1, \ldots, n}$;  
		\item  $  \mathcal{L}og  \; \im (\psi_{| \ker (P_\mathcal{P} \circ (v \times v^*))})$ is the $\QQ\,$-sub-vector space of $\CC / 2 i \pi \QQ$ generated by the logarithms
	 $\{ \log \psi(x,y^\vee ) \}_{(x,y^\vee ) \in \ker (P_{\mathcal{P}_j} \circ (v_j\times v^*_j))  \atop 
		(x,y^\vee ) \in \ZZ^r \otimes \ZZ^s} .$
\end{itemize}

Because of Proposition \ref{proof-periods} and Corollary \ref{eq:dimGalMot}, we can conclude that 

\begin{theorem}\label{thmMain}
 Let $M=[u:{\ZZ}^r \to G]$ be a 1-motive defined over $K$, with $G$ an extension of a product $\Pi^n_{j=1} \cE_j $ of pairwise not isogenous elliptic curves by the torus $\GG_m^s.$ Then the Generalized Grothendieck's Period Conjecture applied to $M$ is equivalent to the 1-motivic elliptic conjecture. 
\end{theorem}

\begin{remark} \label{Rk1}
If $Q_{ji}=0$ for $j=1, \dots,n$ and $i=1, \dots,s$, the above conjecture is the elliptic-toric conjecture stated in \cite[1.1]{B02}, which is equivalent to the Generalized Grothendieck's Period Conjecture applied to the 1-motive $M=[u: \Pi_{k=1}^r z_k {\ZZ} \to \GG_m^s \times \Pi^n_{j=1} \cE_j]$ with $u(z_k) = (R_{1k}, \dots,R_{sk} , P_{1k}, \dots , P_{nk}) \in \GG_m^s (K) \times \Pi^n_{j=1} \cE_j(K).$
\end{remark}

\begin{remark} \label{Rk2}
	If $Q_{ji}=P_{ij}=\cE_j=0$ for $j=1, \dots,n$ and $i=1, \dots,s$, the above conjecture is equivalent to the Generalized Grothendieck's Period Conjecture applied to the 1-motive $M=[u: \Pi_{k=1}^r z_k {\ZZ} \to \GG_m^s]$ with $u(z_k) = (R_{1k}, \dots,R_{sk}) \in \GG_m^s (K) $, which in turn is equivalent to the 
 Schanuel conjecture (see \cite[Cor 1.3 and \S 3]{B02}).
\end{remark}


\section{Low dimensional case: $r=n=s=1$}\label{lowDim}

In this section we work with a 1-motive $M=[ u:\ZZ \rightarrow G], u(1)=R,$ defined over $K$ in which 
 $G$ is an extension of just one elliptic curve $\cE$ by the torus $\GG_m$, i.e. $r=s=n=1$.

Let $g_2=60 \; \mathrm{G}_4$ and $g_3=140 \; \mathrm{G}_6$ with $\mathrm{G}_4$ and $\mathrm{G}_6$ the Eisenstein series relative to the lattice $\Lambda := \HH_1(\cE(\CC),\ZZ)$ of weight 4 and 6 respectively. The field of definition $K$ of the 1-motive  $M=[u:\ZZ \rightarrow G], u(1)=R$ is 
\[\QQ \big( g_2, g_3, Q,R \big).\]
 By Proposition \ref{proof-periods}, the field $K (\mathrm{periods}(M))$ generated over $K$ by the periods of $M$, which are the coefficients of the matrix (\ref{eq:matrix-periods}), is 
\[\QQ \Big(g_2, g_3, Q, R, 2 i \pi,  \omega_1,\omega_2,\eta_1,\eta_2, p,\zeta(p), \eta_1 q - \omega_1 \zeta(q),\eta_2 q - \omega_2 \zeta(q) , \log f_q(p) +l \Big).
\]

$\End(\cE) \otimes_\ZZ \QQ$-linear dependence between the points $P$ and $Q$ and torsion properties of the points $P, Q,R$  affect the dimension of the unipotent radical of $\Galmot (M)$.
By Corollary \ref{eq:dimGalMot} we have the following table concerning the dimension of the motivic Galois group $\Galmot (M)$ of $M$: 

\begin{center}
	\begin{tabular}{|c|c|c|c|c|}
		\hline
		& $\dim \UR (M)$  & $\dim \Galmot (M) $  & $\dim \Galmot (M) $ & $M$   \\
		&  &  $\cE$ CM  &  $\cE$ not CM & \\ \hline
		Q, R torsion & 0 & 2& 4 & $M=[u:\ZZ \rightarrow \cE \times \GG_m]$ \\ 
		($\Rightarrow$ P torsion)  & & & & $ u(1)=(0,1)$ \\ \hline
		P,Q torsion  &1 & 3& 5 &  $M=[u:\ZZ \rightarrow \cE \times \GG_m] $\\ 
		(R not torsion)  & & & &  $u(1)=(0,R) $ \\ \hline
		R torsion   &2 & 4& 6 & $M=[u:\ZZ \rightarrow G]$\\
		($\Rightarrow$ P torsion)  & & & & $ u(1)=0$\\ \hline
		Q torsion  &3 &5 & 7 & $M=[u:\ZZ \rightarrow \cE \times \GG_m]$\\ 
		(P and R not torsion)  & & & & $u(1)=(P,R) $  \\ \hline
		P torsion  &3 &5 &7  &  $M=[u:\ZZ \rightarrow \cE^* \times \GG_m] $ \\ 
		(R and Q not torsion)  & & & & $u(1)=(Q,R) $ \\ \hline
		P,Q   &5 &7 & 9  & $M=[u:\ZZ \rightarrow G]$\\ 
		$\End( \cE) \otimes_\ZZ \QQ$-lin indep  & & & &$ u(1)=R$\\ \hline
	\end{tabular}
\end{center}

We can now state explicitly the Generalized Grothendieck's Period Conjecture (\ref{eq:GCP}) 
for the 1-motives involved on the above table:

\begin{itemize}
	\item $R$ and $Q$ are torsion:
	We work with the 1-motive $M=[u:\ZZ \rightarrow \cE \times \GG_m], u(1)=(0,1)$
	or $M=[0 \to \cE].$ If $\cE$ is not CM, 
	\[\mathrm{tran.deg}_{\QQ}\, \QQ
	\Big( g_2,g_3,\omega_1,\omega_2,\eta_1,
	\eta_2  \Big)\geq 4\]
	that is 4 at least of the 6 numbers $ g_2,g_3,\omega_1,\omega_2,\eta_1, \eta_2$ are algebraically independent over $\QQ$. If $\cE$ is CM, 
	\[\mathrm{tran.deg}_{\QQ}\, \QQ
	\Big( g_2,g_3,\omega_1,\eta_1  \Big)\geq 2\]
	that is 2 at least of the 4 numbers $g_2,g_3,\omega_1,\eta_1$ are algebraically
	independent over $\QQ.$  
	If we assume $g_2,g_3 \in \overline{\QQ},$ we get the Chudnovsky Theorem: $\mathrm{tran.deg}_{\QQ}\, \QQ
	(\omega_1,\eta_1  )=2.$

	\item $P$ and $Q$ are torsion:
	We work with the 1-motive $M=[u:\ZZ \rightarrow \cE \times \GG_m], u(1)=(0,R)$ (this case was studied in the author's Ph.D thesis, see \cite{B02}).
	If $\cE$ is not CM, 
	\[\mathrm{tran.deg}_{\QQ}\, \QQ
	\Big( g_2,g_3,\omega_1,\omega_2,\eta_1,\eta_2,R, \log(R)  \Big)\geq 5\]
	that is 5 at least of the 8 numbers $g_2,g_3,\omega_1,\omega_2,\eta_1,
	\eta_2,R, \log(R)$ are algebraically independent over $\QQ$. If $\cE$ is CM, 
	\[\mathrm{tran.deg}_{\QQ}\, \QQ
	\Big( g_2,g_3,\omega_1,\eta_1,R, \log(R)  \Big)\geq 3\]
	that is 3 at least of the 6 numbers $g_2,g_3,\omega_1,\eta_1,R, \log(R)$ are algebraically independent over $\QQ$.

	\item $R$ is torsion:
	We work with the 1-motive $M=[u:\ZZ \rightarrow G], u(1)=0$
	or $M=[v^*:\ZZ \to \cE^*], v^*(1)=Q.$ If $\cE$ is not CM, 
	\[\mathrm{tran.deg}_{\QQ}\, \QQ
	\Big( g_2,g_3,\omega_1,\omega_2,\eta_1, \eta_2,Q,q, \zeta(q) \Big)\geq 6\]
	that is 6 at least of the 9 numbers $g_2,g_3,\omega_1,\omega_2,\eta_1, \eta_2,Q,q, \zeta(q)$ are algebraically independent over $\QQ$. If $\cE$ is CM, 
	\[\mathrm{tran.deg}_{\QQ}\, \QQ
	\Big(   g_2,g_3,\omega_1,\eta_1,Q, q, \zeta(q)  \Big)\geq 4\]
	that is 4 at least of the 7 numbers $ g_2,g_3,\omega_1,\eta_1,Q, q, \zeta(q)$ are algebraically independent over $\QQ$.

	
	\item $Q$ is torsion:
	We work with the 1-motive  $M=[u:\ZZ \rightarrow \cE \times \GG_m], u(1)=(P,R)$ (this case was considered in the author's Ph.D thesis, see \cite{B02}).
	If $\cE$ is not CM, 
	\[\mathrm{tran.deg}_{\QQ}\, \QQ
	\Big( g_2,g_3,\omega_1,\omega_2,\eta_1, \eta_2,  P,R,p,\zeta(p), \log(R)  \Big)\geq 7\]
	that is 7 at least of the 11 numbers $g_2,g_3,\omega_1,\omega_2,\eta_1, \eta_2,  P,R,p,\zeta(p), \log(R) $ are algebraically independent over $\QQ$. If $\cE$ is CM, 
	\[\mathrm{tran.deg}_{\QQ}\, \QQ
	\Big(g_2,g_3,\omega_1,\eta_1,  P,R,p,\zeta(p), \log(R)  \Big)\geq 5\]
	that is 5 at least of the 9 numbers $g_2,g_3,\omega_1,\eta_1,  P,R,p,\zeta(p), \log(R) $ are algebraically independent over $\QQ$.

	
	\item $P$ is torsion:
	We work with the 1-motive $M=[u:\ZZ \rightarrow G], u(1)=R \in \GG_m(K)$ or $M=[u:\ZZ \rightarrow \cE^* \times \GG_m], u(1)=(Q,R).$ If $\cE$ is not CM, 
	\[\mathrm{tran.deg}_{\QQ}\, \QQ
	\Big(  g_2,g_3,\omega_1,\omega_2,\eta_1,
	\eta_2,Q,R, q, \zeta(q),\log (R)  \Big)\geq 7\]
	that is 7 at least of the 11 numbers $ g_2,g_3,\omega_1,\omega_2,\eta_1,
	\eta_2,Q,R, q, \zeta(q),\log (R)  $ are algebraically independent over $\QQ$. If $\cE$ is CM, 
	\[\mathrm{tran.deg}_{\QQ}\, \QQ
	\Big(  g_2,g_3,\omega_1,\eta_1,Q,R, q, \zeta(q),\log (R)  \Big)\geq 5\]
	that is 5 at least of the 9 numbers $ g_2,g_3,\omega_1,\eta_1,Q,R, q, \zeta(q),\log (R)  $ are algebraically independent over $\QQ$.  
	
	
	\item $P,Q,R$ are not torsion and $P,Q$ are $\End( \cE) \otimes_\ZZ \QQ$-linearly independent:
	We work with the 1-motive $M=[u:\ZZ \rightarrow G], u(1)=R \in G(K).$ 
	If $\cE$ is not CM, 
	\[\mathrm{tran.deg}_{\QQ}\, \QQ
	\Big(  g_2, g_3, Q, R, \omega_1,\omega_2,\eta_1,\eta_2, p,\zeta(p),q, \zeta(q), \eta_1 q - \omega_1 \zeta(q),\eta_2 q - \omega_2 \zeta(q) , \log f_q(p)+l  \Big)\geq 9\]
	that is 9 at least of the 15 numbers $ g_2, g_3, Q, R, \omega_1,\omega_2,\eta_1,\eta_2, p,\zeta(p),q, \zeta(q), \eta_1 q - \omega_1 \zeta(q),\eta_2 q - \omega_2 \zeta(q) , \log f_q(p) $ are algebraically independent over $\QQ$. If $\cE$ is CM, 
	\[\mathrm{tran.deg}_{\QQ}\, \QQ
	\Big( g_2, g_3, Q, R, \omega_1,\eta_1, p,\zeta(p),q, \zeta(q), \eta_1 q - \omega_1 \zeta(q),\eta_2 q - \omega_2 \zeta(q) , \log f_q(p) +l  \Big)\geq 7\]
	that is 7 at least of the 13 numbers $ g_2, g_3, Q, R, \omega_1,\eta_1, p,\zeta(p),q, \zeta(q), \eta_1 q - \omega_1 \zeta(q),\eta_2 q - \omega_2 \zeta(q) , \log f_q(p)   $ are 
	algebraically independent over $\QQ$.  
	
\end{itemize}



\bibliographystyle{plain}

\newpage
\section*{Letter of Y. Andr\'e}

 \rightline{Paris, 29 may 2019}

\bigskip 
\bigskip 
{Dear Cristiana},

\bigskip
Following your query, I will try to summarize the formalism of Grothendieck's period conjecture, present different variants,  sketch their relations and give some historical hints and references. 

\smallskip {\bf Origins}. Grothendieck's period conjecture deals with transcendence properties of periods of algebraic varieties defined over a number field. In essence, it predicts that algebraic relations between periods come from geometry. Its first mention appears as a footnote in Grothendieck's letter to Atiyah (Publ. IHES 29, 1966) where, after mentioning Schneider's results on elliptic periods, he alludes to the existence of a general conjecture. A first published statement is contained in Lang's book on transcendental number theory (Addison-Wesley 1966, chap. 4, historical note). The next published related statement, without mention of Grothendieck's name/conjecture, is at the beginning of Deligne's paper ``Hodge cycles on abelian varieties"  (Springer LN 900, 1982; see also the end of its announcement, Bull. SMF. 1980). 

The next published statement, and explanation of the relationship between the previous statements, is in chapter 9 of my book on G-functions (Vieweg 1989, recently reprinted by Springer), entitled ``towards Grothendieck's conjecture on periods of algebraic manifolds". A more complete exposition of the formalism, and its relation to a fullness conjecture of enriched realization of motives (parallel to the Hodge or Tate conjectures) is discussed in my SMF book on motives, denoted henceforth [IM] (2004); related material forms the whole third part of that book. 
In [IM, 23.4-5] and in other contemporary papers, I extended the period conjecture in two directions: $i)$ the idea of Galois theory of periods, $ii)$ the generalization of the period conjecture for motives defined over an arbitrary subfield of $\bf C$.

A different but related thread came with Kontsevich's period conjecture (preliminary version by the SMF, 1998; final version with Zagier: ``periods", Springer 2001). Among other things, he conjectured that algebraic relations between periods come from the formal properties of $\int$, and indicated (relying on Nori's work) that this conjecture is equivalent to Grothendieck's period conjecture for all motives over number fields.

\smallskip {\bf Motivic Galois groups}. In order to express appropriately his intuition that algebraic relations between periods should come from geometry, Grothendieck uses his idea of motives and motivic Galois theory\footnote{this statement is not just a conjecture in the history of mathematics: a decade ago in Montpellier, I had the priviledge to consult some unpublished notes by Grothendieck on motives (which by now may be online: (https://grothendieck.umontpellier.fr/archives-grothendieck/, cotes 10-19 - thanks to J. Fresan for the reference), and I saw that he really wrote the period conjecture essentially as formulated below.}. At Grothendieck's time, however, this theory was only a dream (with precise contours), so that the period conjecture was more a metaconjecture than a conjecture (in the sense that some terms were not well-defined); but some consequences of the conjecture could be formulated in well-defined terms, and all ensuing statements, however remote from the original intuition, were called "period conjectures", creating many an ambiguity. 

Nowadays, there is an unconditional tannakian category of motives over any field $k$ of characteristic zero, which is of ``geometric" nature (in the sense that the morphisms arise ``somehow" from algebraic correspondences), so that the conjecture can be neatly stated, and reflects the original intuition. In fact, three such theories have been constructed: the first (restricted to pure motives, i.e. motives coming from projective smooth $k$-varieties) was defined in my IHES paper (1996); the second by Nori (unpublished notes have circulated, and there is now the book by Huber and M\"uller-Stach, Springer 2017), the third by Ayoub (he defines a motivic Galois group using Voevodsky's triangulated motives). These constructions look quite different, but turn out to be compatible (as shown by Arapura, Choudhury/Gallauer): in short, Nori's and Ayoub's absolute mixed motivic Galois groups over $k$ are "the same", and my absolute pure motivic Galois group is just its pro-reductive quotient. In the sequel, I will thus speak about ``the" motivic Galois group $G_{mot}(M)$ of a motive $M$ defined over $k\subset \bf C$, without being more explicit: this is a well-defined linear algebraic group defined over $\bf Q$ (a closed subgroup of the group of linear automorphisms of the Betti realization of $M$), which is reductive if $M$ is pure. The absolute motivic Galois group of $k$ is ``their projective limit" for various $M$.

These groups contain other previously defined groups: $$G_{mot}(M)\supset MTA(M) \supset MT(M),$$ where $MT(M)$ is the Mumford-Tate group attached to (the mixed Hodge structure of) $M_{\bf C}$, and $MTA(M)$ is the absolute Mumford-Tate group defined by Deligne (a.k.a. ``absolute Hodge motivic Galois group"). The definition and computation of $MT(M)$ being easier than the others, it is interesting to know when these groups coincide. When $M$ is an abelian variety or a $1$-motive, defined over an algebraically closed $k \subset \bf C$, Deligne and Brylinsky proved that $MTA(M) \supset MT(M)$ is an equality, and I later proved the stronger statement that $G_{mot}(M) = MT(M)$ in those cases (Imrn 2019). 

\smallskip {\bf Period torsors}. For any (pure or mixed) motive $M$ over $k \subset \bf C$, $G_{mot}(M)$ is defined as the group scheme of $\otimes$-automorphisms of the Betti realization of the tannakian category $\langle M\rangle$ generated by the motive $M$.  One may also consider the algebraic De Rham realization, with values in $k$-vectors spaces. The scheme of $\otimes$-isomorphisms from De Rham to Betti$\otimes k$ is a torsor under 
$G_{mot}(M)_k$, the period torsor $\Pi(M)$. The name comes from the fact that integration gives rise to the ``period isomorphism" $$\int: \,H_{dR}(M)\otimes_k {\bf C} \cong H_B(M)\otimes_{\bf Q}\bf C$$ (concretely, a matrix with entries the periods of $M$), and further to a canonical $\bf C$-point of $\Pi(M)$:
 $$\varpi:  {\rm{Spec}}\,{\bf C}\to \Pi(M).$$

\smallskip {\bf Grothendieck's period conjecture.} $(?) $ {\it If $k\subset \bar{\bf Q}$, then $\varpi$ maps to the generic point of $\Pi(M)$. }

In more heuristic words, {\it the periods of $M$ generate the $k$-algebra of functions of the $k$-variety $\Pi(M)$}, or else: {\it the algebraic relations between periods of $M$ come from the morphisms in $\langle M\rangle$} (which are of ``geometric" nature).  

I insist that this should hold for any motive (pure or mixed) defined over any algebraic field $k$.

 The conjecture includes the subconjecture that $\Pi(M)$ is irreducible (={connected}, since this is a torsor). In fact, it is equivalent to the connectedness of $\Pi(M)$, plus equality of dimensions: the dimension of the $k$-Zariski closure of the image of $\varpi$ is the dimension of $\Pi(M)$. By the relation between dimension and transcendence degree in commutative algebra, the former dimension is nothing but the transcendence degree over $k$ (or $\bf Q$) of the $k$-subalgebra of $\bf C$ generated by the periods of $M$, and the latter dimension is ${\rm{dim}}\, G_{mot}(M)$ since $\Pi(M)$ is a torsor under $G_{mot}(M)_k$. 
  Therefore $(?)$ is equivalent to the {\it connectedness of $\Pi(M)$}, plus the equality:
$$(??)\;\;\; {\rm{transc. deg}}_{\bf Q}\, k({\rm{periods}}(M)) = {\rm{dim}}\, G_{mot}(M).$$
 This formulation is more congenial to transcendental number theorists (provided of course that one knows how to calculate or at least estimate the right hand side), while $(?)$ is more geometric. In some ``applications" of the period conjecture, it may be necessary to take into account the geometry of the period torsor, and not just the numerical identity $(??)$, cf. e.g. [IM 23.2]. 
 
 \smallskip The formulation given in Lang's book is the following: {assume that $M$ is the motive of a projective smooth $k$-variety $X$, and note that any algebraic cycle on $X^n, n\in \bf N,$ has a De Rham class in $H_{dR}(X^n)=H_{dR}(X)^{\otimes n}$ and a Betti class in $H_{B}(X^n)=H_{B}(X)^{\otimes n}$, hence gives rise via $\int$ to polynomial relations of degree $n$ between periods of $M$; the conjecture predicts that 

\centerline{$(???)\;\;\; $ {\it these relations generate an ideal of definition for the period matrix of $M$.}} 
\noindent If one assumes Grothendieck's standard conjecture\footnote{usually, one writes: standard conjecture{\it s}, but in characteristic $0$ they amount to one single statement, cf. e.g. [IM, chap. 5].}, $G_{mot}(M)$ is the group which fixes the classes of algebraic cycles in tensor powers of $H_B(M)$, and parallely, the previous period relations are equations for $\Pi(M)$. It is not difficult to deduce from there that {\it in the pure case, $(???)$ is equivalent to $(?)$ plus the standard conjecture}. 
 
  \smallskip The relation with Kontsevich's period conjecture becomes apparent if one considers all motives together, and not just $\langle M\rangle  $\footnote{however, given your special interest in the case of $1$-motives, let me mention the recent work of Huber and W\"ustholz, who manage to formulate a period conjecture in Kontsevich's style just for $1$-motives.}.  

\smallskip {\bf Relation to fullness of enriched realizations.} A natural framework to deal with period problems is the tannakian category $Vec_{k, \bf Q}$ (appearing in [IM, 7.5]) consisting of a $k$-vector space, a $\bf Q$-vector space, and an isomorphism between their complexifications. De Rham and Betti realizations, together with $\varpi$, give rise to a $\otimes$-functor $dRB$ from $\langle M\rangle  $ to $Vec_{k, \bf Q}$. The period conjecture $(?)$ implies that this functor is {\it full}. But fullness of $dRB$ is a much weaker conjecture\footnote{for instance, if $M$ is the motive of an elliptic curve, fullness follows from known results in transcendental number theory, while $(?)$ is known only in the presence of complex multiplication. Another illustration of the difference arises if one considers all abelian varieties with complex multiplication by a cyclotomic field: fullness of $dRB$ (resp. period conjecture) is equivalent in this case to Rohrlich's (resp. Lang's) conjecture that all monomial (resp. algebraic) relations between special values of the gamma function come from the functional relations. There are several more recent results in the spirit of this fullness conjecture (Andreatta/Barbieri-Viale/Bertapelle, Huber/W\"ustholz, Kahn, myself), but virtually nothing new about $(?)$.}. 

Let $G_{k, \bf Q}(M)$ be the tannakian group attached to $dRB(M)\in Vec_{k, \bf Q}$, a group defined purely in terms of the periods of $M$. One has $MTA(M)\supset G_{k, \bf Q}(M)$}, and the image of $\varpi$ lies in a $G_{k, \bf Q}(M)$-torsor contained in $\Pi(M)$. The fullness of $dRB$ follows if this torsor is $\Pi(M)$ itself (and conversely, if $G_{k, \bf Q}(M)$ is a so-called observable subgroup of $G_{mot}(M)$).

 \smallskip {\it Remark}. In $(??)$, inequality $\leq$ is unconditional. Moreover, it also holds with $G_{mot}(M)$ replaced by $MTA(M)$ or $G_{k, \bf Q}(M)$. It follows that $(?)$ implies $G_{mot}(M) = MTA(M)= G_{k, \bf Q}(M)$. In fact, $(??)$ splits into two equalities 
 $$(??)'\;\;\; {\rm{transc. deg}}_{\bf Q}\, k({\rm{periods}}(M)) = {\rm{dim}}\, G_{k, \bf Q}(M) = {\rm{dim}}\, G_{mot}(M)$$
 which can be studied separately.  
 On the other hand, one can weaken inequality $\geq$ in $(??)$ on replacing $G_{mot}(M)$ by $MT(M)$ (but the reverse inequality becomes unclear; and in doing so, one 
 looses the essence of the conjecture, which is that relations between periods should come from geometry\footnote{but in the special case of $1$-motives, one actually looses nothing by my aformentioned results.}). What is coined ``motivic periods" or ``formal periods" in the literature refers essentially to coordinates on $\Pi(M)$, or (less appropriately) on the corresponding torsors under $MTA(M)$ or $G_{k, \bf Q}(M)$, depending on the context; of course, under $(?)$, they may be ``identified" with actual periods.

\smallskip {\bf (Generalized) period conjecture over an arbitrary subfield of $\bf C$.} The first published version of this conjecture (which I made around 1997) is [IM, 23.4]. In analogy with $(??)$, it predicts that for any $k\subset \bf C$, and any (pure or mixed) motive $M$ defined over $k$,
$$(?!)\;\;\; {\rm{transc. deg}}_{\bf Q}\, k({\rm{periods}}(M)) \geq {\rm{dim}}\, G_{mot}(M).$$ Of course, since $k$ may contain the periods, one cannot hope for an equality. 
The first test that I made before stating it was the case of $1$-motives without abelian part, in which case one recovers Schanuel's conjecture. Later, you studied many other cases in detail and gave evidence that this conjecture looks sharp and might be optimal. 

On the other hand, I am not aware of any ``geometric" version of this conjecture in the style of $(?)$\footnote{there are many other related open questions; e.g. recent work by Fresan and Jossen, following an intuition of Kontsevich, has shaped the contours of a theory of ``exponential motives". A period conjecture in the style of $(??)$ may hold for them. Does it follow from $(?!)$?}.

\smallskip {\bf Functional analog.} To be complete, I ought to discuss the long story of the functional analog of Grothendieck's period conjecture and Ayoub's work which settles it (Ann. Maths 2015). But this letter is already too long, and as this lies beyond your query, I will content myself with the following indications. If $M$ is a motive over a function field $k$ in one variable over $\bf C$, one can define periods of $M$ as elements of the completion $\hat k$ at any place of good reduction for $M$.  These periods are solutions of a Picard-Fuchs differential equation (a linear differential equation with coefficients in $k$). One can define a period torsor $\Pi(M)$ in this context, which is a torsor under the differential Galois group (= algebraic monodromy group, since the singularities are regular), as well as a canonical $\hat k$-point. By Kolchin's theorem, the image of this point is Zariski-dense in $\Pi(M)$. It remains to relate the monodromy group to motives. This has been done by Ayoub\footnote{there are unpublished similar works by Nori and by Jossen; and a related published result by Arapura, Adv. Math 233 (2013).}: here, the algebraic mondromy group coincides with the {\it relative motivic Galois group}, i.e. the kernel of the map $G_{mot}(M)\to G_{mot, cst}(M)$ dual to the inclusion of the category of constant motives inside $\langle M\rangle$.

 My Bourbaki survey (1995) touches all the subjects of this letter in greater detail.

\medskip With my best wishes,  

\bigskip
\bigskip
                      \centerline{       Yves.}

\newpage

\section*{Appendix by M. Waldschmidt: Third kind elliptic integrals and transcendence}

This short appendix aims at giving references on papers related with transcendence results concerning elliptic integrals of the third kind. So far, results on transcendence and linear independence are known, but there are very few results on algebraic independence.

 \medskip
 
 In his book on transcendental numbers \cite{Schneider}, Th.~Schneider proposes eight open problems, the third of which is : 
{\it Try to find transcendence results on elliptic integrals of the third kind.}
 
 In \cite[Historical Note of Chapter IV]{Lang}, S.~Lang explains the connections between elliptic integrals of the second kind, Weierstrass zeta function and extensions of an elliptic curve by $\GG_a$. He applies the so--called Schneider--Lang criterion to the Weierstrass elliptic and zeta functions and deduces the transcendence results due to Th.~Schneider on elliptic integrals of the first and second kind. At that time, it was not known how to use this method for proving results on elliptic integrals of the third kind. 
 
 The solution came from \cite{Serre}, where J-P.~Serre introduces the functions $f_q$ (with the notation of \cite{Bertolin}) related to elliptic integrals of the third kind, which satisfy the hypotheses of the Schneider-Lang criterion and are attached to extensions of an elliptic curve by $\GG_m$. This is how the first transcendence results on these integrals were obtained \cite{MWAsterisque,MWCanada}. In \cite{BertrandLaurent}, D.~Bertrand and M.~Laurent give further applications of the Schneider-Lang criterion involving elliptic integrals of the third kind. Applications are given in \cite{Bertrand1983a,Bertrand1983b,Scholl1986}, dealing with the Neron--Tate canonical height on an elliptic curve (including the $p$--adic height) and the arithmetic nature of Fourier coefficients of Eisenstein series. A first generalization to abelian integrals of the third kind is quoted in \cite{Bertrand1983b}. Transcendence measures are given in \cite{Reyssat1980AnnToulouse}.
 
  Properties of the smooth Serre compactification of a commutative algebraic group and of the exponential map, together with the links with integrals, are studied in \cite{FaltingsWustholz1984}. 
  See also \cite{KnopLange1985}. 
 In \cite[Chapter 20 -- Elliptic functions]{Masser2016} (see in particular Theorem 20.11 and exercises 20.104 and 20.105) more details are given on the functions associated with elliptic integrals of the third kind, the associated algebraic groups, which are extensions of an elliptic curve by $\GG_m$, and the consequences of the Schneider-Lang criterion.
 
The first results of linear independence of periods of elliptic integrals of the third kind are due to M.~Laurent \cite{LaurentCrelle1980,LaurentCrelle1982} (he announced his results in 
 \cite{LaurentCRAS,LaurentDPP}). The proof uses Baker's method. More general results on linear independence are due to G.~W\"ustholz \cite{WustholzCrelle1984} (see also \cite[
 \S~6.2]{BakerWustholz2007}), including the following one, which answers a conjecture that M.~Laurent stated in \cite{LaurentCrelle1982} where he proved special cases of it.
 Let $\wp$ be a Weierstrass elliptic function with algebraic invariants $g_2$, $g_3$. Let $\zeta$ be the corresponding Weierstrass zeta function, $\omega$ a nonzero period of $\wp$ and $\eta$ the corresponding quasi-period of $\zeta$. Let $u_1,\ldots,u_n$ be complex numbers which are not poles of $\wp$, which are $\QQ$ linearly independent modulo $\ZZ\omega$ and such that $\wp(u_1),\dots,\wp(u_n)$ are algebraic. Define
 $$
 \lambda(u_i,\omega)=\omega\zeta(u_i)-\eta u_i.
$$ 
 Then 
 the $n+3$ numbers 
 $$
 1,\omega,\eta,\lambda(u_1),\ldots,\lambda(u_n)
$$
are linearly independent over $\Qbar$. 
 
 The question of the transcendence of the nonvanishing periods of a meromorphic differential form on an elliptic curve defined over the field of algebraic numbers is now solved \cite[Theorem 6.6]{BakerWustholz2007}. See also 
\cite{HW2018}, as well as \cite[\S~1.5]{Tretkoff2017} for abelian integrals of the first and second kind.
A reference of historical interest to a letter from Leibniz to Huygens in 1691 is quoted in 
 \cite[\S~6.3]{BakerWustholz2007} 
 and \cite{Wustholz2012}.

The only results on algebraic independence related with elliptic integrals of the third kind so far are those obtained by ƒ.~Reyssat \cite{Reyssat1980Cras,Reyssat1982} and by R.~Tubbs \cite{Tubbs1987,Tubbs1990}. We are very far from anything close to the conjectures in 
\cite{Bertolin}.

For a survey (with an extensive bibliography including 254 entries), see \cite{MW2008}.
 


The references below are listed by chronological order.

\bibliographystyle{plain}

\end{document}